\documentclass[letterpaper,11pt]{article}

\usepackage{amsmath}
\usepackage{amsthm}
\usepackage{authblk}
\usepackage{mathpazo}
\usepackage{palatino}
\usepackage{setspace}
\usepackage{url}
\usepackage[comma,authoryear]{natbib}
\usepackage{hyperref}
\usepackage{theoremref}
\usepackage[titletoc,title]{appendix}
\usepackage[table]{xcolor}
\usepackage{float}
\restylefloat{table}
\usepackage{amssymb}
\usepackage{caption}
\usepackage{subcaption}
\usepackage[shortlabels]{enumitem}

\definecolor{armygreen}{rgb}{0.29, 0.33, 0.13}
\definecolor{auburn}{rgb}{0.43, 0.21, 0.1}
\definecolor{burgundy}{rgb}{0.5, 0.0, 0.13}
\definecolor{medium red}{rgb}{.490,.298,.337}
\definecolor{dark red}{rgb}{.235,.141,.161}
\definecolor{dark green}{rgb}{0.0,0.5,0.0}

\hypersetup{
	colorlinks = true,
	linkcolor = {burgundy},
	urlcolor = {burgundy},
	citecolor = {burgundy},		
	linkbordercolor = {white},
}

\newtheorem{theorem}{Theorem}[section]

\newtheorem{proposition}{Proposition}[section]
\newtheorem{claim}{Claim}[section]
\newtheorem{lemma}{Lemma}[section]
\newtheorem{corollary}{Corollary}[section]

\theoremstyle{definition}
\newtheorem{definition}{Definition}[section]
\theoremstyle{definition}
\newtheorem{example}{Example}[section]
\theoremstyle{definition}
\newtheorem{remark}{Remark}[section]
\theoremstyle{definition}
\newtheorem{note}{Note}[section]
\theoremstyle{definition}

\newcommand*{\claimproofname}{Claim proof}
\newenvironment{claimproof}[1][\claimproofname]{\begin{proof}[#1]}{\end{proof}}

\usepackage{verbatim}
\usepackage{tikz}
\tikzset{
	treenode/.style = {shape=rectangle, rounded corners,
		draw, align=center,
		top color=white, bottom color=blue!20},
	root/.style  = {shape=rectangle, rounded corners,
		draw, align=center,
		top color=white, bottom color=blue!20},
	root1/.style  = {shape=rectangle, rounded corners,
		draw, align=center,
		top color=white, bottom color=red!40},
	root2/.style  = {shape=rectangle, rounded corners,
		draw, align=center,
		top color=white, bottom color=red!20},
	env/.style  = {shape=rectangle, rounded corners,
		draw, align=center,
		top color=white, bottom color=blue!20},
}

\title{On percolation in a generalized backbend process\thanks{We thank Kumarjit Saha for helpful discussions.}}

\author{Pinaki Mandal\thanks{Corresponding author: Economic Research Unit, Indian Statistical Institute, Kolkata, India. E-mail: pnk.rana@gmail.com} \; and Souvik Roy\thanks{Economic Research Unit, Indian Statistical Institute, Kolkata, India. E-mail: souvik.2004@gmail.com.}}

\date{}

\begin{document} 	
	\maketitle

	\begin{abstract}
		We have generalized the idea of backbend in a nearest-neighbor oriented bond percolation process by considering a backbend sequence $\beta : \mathbb{Z}_+ \to \mathbb{Z}_+ \cup \{\infty\}$, and defining a $\beta$-backbend path from the origin as a path that never retreats further than $\beta(h)$ levels back from its record level $h$. We study the relationship between the critical probabilities of different percolation processes based on different backbend sequences on half-space, full-space, and half-slabs of the $d$-dimensional ($d \geq 2$) body-centered cubic (BCC) lattice. We also give sufficient conditions on the backbend sequences such that there will be no percolation at the critical probabilities of the corresponding percolation processes on half-space and full-space of the BCC lattice.	
	\end{abstract}

	\newpage

	\section{Introduction}\label{section intro}

	In this paper, we discuss an extension of the concept of nearest-neighbor oriented bond percolation with backbend on certain sublattices of the $d$-dimensional ($d \geq 2$) body-centered cubic (BCC) lattice. \citet{roy1998backbends} studied nearest-neighbor oriented bond percolation with backbend in the $d$-dimensional simple cubic lattice, where they worked with \textit{$b$-backbend paths} for some $b \in \mathbb{Z}_+$.\footnote{We denote by $\mathbb{Z}_+$ the set $\mathbb{N} \cup \{0\}$.} A $b$-backbend path never retreats further than $b$ levels back from its record level. We have generalized their idea of backbend as follows: instead of taking a constant value $b$ for permissible backbend of a path, we consider a \textit{backbend sequence} $\beta : \mathbb{Z}_+ \to \mathbb{Z}_+ \cup \{\infty\}$, and define a $\beta$-backbend path from the origin as a path that never retreats further than $\beta(h)$ levels back from its record level $h$. The generality of our framework allows us to embed important special cases; a $\beta$-backbend path can be an oriented path, a $b$-backbend path, or an ordinary (unoriented) path depending on the backbend sequence $\beta$.
	
	We study the relationship between the critical probabilities of different percolation processes based on different backbend sequences on half-space, full-space, and half-slabs of the BCC lattice. We also give sufficient conditions on the backbend sequences such that there will be no percolation at the critical probabilities of the corresponding percolation processes on half-space and full-space of the BCC lattice.

	\section{Model}\label{section model}

	We denote the $d$-dimensional ($d \geq 2$) BCC lattice by $(\mathbb{V}, \mathbb{E})$, where the vertex set $\mathbb{V} := \big\{x \in \mathbb{Z}^d : x_i \equiv x_j \mbox{ (mod 2)}$ for all $i, j \in \{1, \ldots, d\} \big\}$ and the edge set $\mathbb{E}$ contains all \emph{\textbf{unordered}} pairs $\langle x, y \rangle$ of vertices $x, y \in \mathbb{V}$ with $\vert x_1 - y_1 \vert = \cdots = \vert x_d - y_d \vert = 1$.\footnote{If we rotate the square lattice using the following rotational matrix $$\begin{bmatrix} 1 & -1 \\ 1 & 1 \\ \end{bmatrix},$$ we will get the 2-dimensional BCC lattice. Thus, $2$-dimensional BCC lattice is equivalent to the square lattice (see \citet{grimmett2013percolation} for the formal definition of the square lattice).} Each edge is independently \emph{\textbf{open}} with probability $p \in [0,1]$ and \emph{\textbf{closed}} with probability $1-p$. Let $\mathbb{P}_p$ denote the corresponding probability measure for the total configurations of all the edges.\footnote{More formally, we consider the following probability space. As \emph{sample space} we take $\Omega = \underset{e \in \mathbb{E}}{\prod} \hspace{1 mm} \{0,1\}$, points of which are represented as $\omega = \{\omega(e) : e \in \mathbb{E}\}$ and called \emph{configurations}; the value $w(e) = 0$ corresponds to $e$ being closed, and $w(e) = 1$ corresponds to $e$ being open. We take $\mathcal{F}$ to be the $\sigma$-algebra of subsets of $\Omega$ generated by the finite-dimensional cylinders. Finally, we take product measure with density $p$ on $(\Omega, \mathcal{F})$; this is the measure $$\mathbb{P}_p = \underset{e \in \mathbb{E}}{\prod} \hspace{1 mm} \mu_e$$ where $\mu_e$ is Bernoulli measure on $\{0,1\}$, given by $\mu_e(\omega(e) = 0) = 1-p$ and $\mu_e(\omega(e) = 1) = p$.}
	
	For $\widehat{\mathbb{V}} \subseteq \mathbb{V}$, a \emph{\textbf{path in $\widehat{\mathbb{V}}$}} is a sequence $x^0, x^1, \ldots, x^n$ of \emph{distinct} vertices such that $x^i \in \widehat{\mathbb{V}}$ and $\langle x^i, x^{i+1} \rangle \in \mathbb{E}$.
	Note that paths are self-avoiding by definition. We call a path \emph{\textbf{open}} if all of its edges are open.
	For a path $\pi = (x^0, \ldots, x^n)$, the \textbf{\textit{record level}} attained by the path $\pi$ till $x^i$, denoted by $h^i(\pi)$, is defined as $\underset{j \leq i}{\max} \hspace{1 mm} x^j_d$. In words, the record level attained by the path $\pi$ till $x^i$ is the maximum value of the $d$-th coordinates of vertices till $x^i$. 
	
	A \textbf{\textit{backbend sequence}} is a mapping $\beta : \mathbb{Z}_+ \to \mathbb{Z}_+ \cup \{\infty\}$. For ease of presentation, we write $\beta(i)$ as $\beta_i$. We introduce the notion of a $\beta$-backbend path.

	\begin{definition}\label{def beta backbend path}
		For $\widehat{\mathbb{V}} \subseteq \mathbb{V}$, $x \in \widehat{\mathbb{V}}$ with $x_d \geq 0$, and a backbend sequence $\beta$, we say a path $\pi = (x^0 = x, \ldots, x^n)$ in $\widehat{\mathbb{V}}$ is a \emph{\textbf{$\beta$-backbend path in $\widehat{\mathbb{V}}$ from $x$ to $x^n$}} if for every $i$,
		\begin{equation*}
			x^i_d \geq h^i(\pi) - \beta_{h^i(\pi)}.
		\end{equation*}
	\end{definition}

	In words, a $\beta$-backbend path never retreats further than $\beta_h$ levels back from its record level $h$. 
	Thus, if $\beta_i = 0$ for all $i \in\mathbb{Z}_+$, then the $\beta$-backbend path is an oriented path in the direction of the positive $d$-th coordinate axis (that is, in the direction of $(0, \ldots, 0, 1)$). If $\beta_i = b$ for all $i \in\mathbb{Z}_+$ and for some $b \in \mathbb{N}$, then the $\beta$-backbend path is a $b$-backbend path in the direction of the positive $d$-th coordinate axis.\footnote{See \citet{roy1998backbends} for the definition of a $b$-backbend path in the context of $d$-dimensional simple cubic lattice. Note that in their paper, the orientation is in the direction of $(1, \ldots, 1)$.} Finally, if $\beta_i = \infty$ for all $i \in\mathbb{Z}_+$, then the $\beta$-backbend path is an ordinary (unoriented) path.
	
	A backbend sequence $\beta$ is \textbf{\textit{$k$-cyclic}} for some $k \in \mathbb{N}$, if for all $i \in \{0, \ldots, k-1\}$, we have $\beta_{kn+i} = \beta_i \in \mathbb{Z}_+$ for all $n \in \mathbb{N}$ (that is, $\beta_i = \beta_{k+i} = \beta_{2k+i} = \cdots$ and $\beta_i \in \mathbb{Z}_+$). Note that a $1$-cyclic backbend sequence is a constant sequence in $\mathbb{Z}_+$.
	Let $\tilde{\beta}$ be a backbend sequence. We say that $\tilde{\beta}$ \textbf{\textit{converges to a $k$-cyclic backbend sequence $\beta$}} for some $k \in \mathbb{N}$, if for all $i \in \{0, \ldots, k-1\}$ we have $\underset{n \rightarrow \infty}{\lim} \hspace{1 mm} \tilde{\beta}_{kn+i} = \beta_i$, and say that $\tilde{\beta}$ \textbf{\textit{converges from below to a $k$-cyclic backbend sequence $\beta$}} for some $k \in \mathbb{N}$, if for all $i \in \{0, \ldots, k-1\}$ we have $\underset{n \rightarrow \infty}{\lim} \hspace{1 mm} \tilde{\beta}_{kn+i} = \beta_i$ and $\tilde{\beta}_{kn+i} \leq \beta_i$ for all $n \in \mathbb{Z}_+$.\footnote{Note that $\tilde{\beta}$ converges from below to a $k$-cyclic backbend sequence does \textit{not} mean that for any $i \in \{0,\ldots, k-1\}$, the subsequence $(\tilde{\beta}_{kn+i})_{n \in \mathbb{Z}_+}$ is monotonically increasing.} 
	For ease of reference, we call a backbend sequence \textit{$k$-cyclic in the limit} if it converges to a $k$-cyclic backbend sequence, and \textit{$k$-cyclic in the limit from below} if it converges from below to a $k$-cyclic backbend sequence.
	Note that a backbend sequence is $1$-cyclic in the limit if and only if it converges to a finite integer.

	Throughout this paper, we refer to the vertex $(0, \ldots, 0)$ as the \textit{origin}. For $\widehat{\mathbb{V}} \subseteq \mathbb{V}$ and a backbend sequence $\beta$, define the following random set of vertices:
	\begin{equation*}
		C^{\beta}_{\widehat{\mathbb{V}}} := \{x \in \widehat{\mathbb{V}} : \mbox{there is an open $\beta$-backbend path in $\widehat{\mathbb{V}}$ from the origin to $x$}\}.
	\end{equation*}
	
	For $p \in [0,1]$, $\widehat{\mathbb{V}} \subseteq \mathbb{V}$, and a backbend sequence $\beta$, the \textbf{\textit{percolation probability}} is defined as
	\begin{equation*}
		\theta^{\beta}_{\widehat{\mathbb{V}}}(p) := \mathbb{P}_p \big(|C^{\beta}_{\widehat{\mathbb{V}}}| = \infty \big).
	\end{equation*}
	
	For $\widehat{\mathbb{V}} \subseteq \mathbb{V}$ and a backbend sequence $\beta$, the \textbf{\textit{critical probability}} is defined as
	\begin{equation*}
		p^{\beta}_c (\widehat{\mathbb{V}}) := \sup \hspace{1 mm} \{p : \theta^{\beta}_{\widehat{\mathbb{V}}}(p) = 0\}.
	\end{equation*}
	
	As we have explained earlier, by the definition of a $\beta$-backbend path, the critical probability $p^{\beta}_c (\widehat{\mathbb{V}})$ becomes the same as the critical probability of the (i) oriented percolation process in the direction of the positive $d$-th coordinate axis when $\beta_i = 0$ for all $i \in\mathbb{Z}_+$, (ii) $b$-backbend percolation process in the direction of the positive $d$-th coordinate axis when $\beta_i = b$ for all $i \in\mathbb{Z}_+$ and for some $b \in \mathbb{N}$, and (iii) ordinary (unoriented) percolation process when $\beta_i = \infty$ for all $i \in\mathbb{Z}_+$. 
	
	For $\widehat{\mathbb{V}} \subseteq \mathbb{V}$, we denote the critical probability of the oriented percolation process on $\widehat{\mathbb{V}}$ by $p^{0}_c (\widehat{\mathbb{V}})$, and the critical probability of the ordinary percolation process on $\widehat{\mathbb{V}}$ by $p_c (\widehat{\mathbb{V}})$.

	\section{Results}

	Let $\mathbb{H} := \{x \in \mathbb{V} : x_d \geq 0\}$ denote the \textit{half-space}. For all $l \in \mathbb{N}$ and all $e \in \{2, \ldots, d\}$, let $\mathbb{Q}_l^e := \{x \in \mathbb{V} : x \in [-l,l]^{d-e} \times \mathbb{Z}^{e-1} \times \mathbb{Z}_+\}$ denote an \textit{$e$-dimensional half-slab}.\footnote{Note that for all $l \in \mathbb{N}$, $\mathbb{Q}_l^d = \mathbb{H}$.} For all $t \in \mathbb{N}$, let $\mathbb{S}_t := \{x \in \mathbb{V} : 0 \leq x_d \leq t\}$ denote a \textit{$d-1$-dimensional slab}. Furthermore, for all $l_1, \ldots, l_d, r_1, \ldots, r_d \in \mathbb{Z} \cup \{-\infty, \infty\}$ with $l_i \leq r_i$ for all $i = 1, \ldots, d$, let us define
	\begin{equation*}
		B \Big(\underset{i=1}{\overset{d}{\prod}}[l_i, r_i] \Big) := \Big\{x \in \mathbb{V} : x \in \underset{i=1}{\overset{d}{\prod}}[l_i, r_i] \Big\}.
	\end{equation*}
	For ease of presentation, we denote $B \Big(\underset{i=1}{\overset{d-1}{\prod}}[l_i, r_i] \times [l,l] \Big)$ by $B \Big(\underset{i=1}{\overset{d-1}{\prod}}[l_i, r_i] \times l \Big)$.
	
	Furthermore, following our notational terminology, for $\widehat{\mathbb{V}} \subseteq \mathbb{V}$, $\mathbb{V}' \subseteq \mathbb{H}$, and a backbend sequence $\beta$, define the random set of vertices 
	\begin{equation*}
		C^{\beta}_{\widehat{\mathbb{V}}}(\mathbb{V}') := \Big\{y \in \widehat{\mathbb{V}} : \mbox{ there is an open $\beta$-backbend path in } \widehat{\mathbb{V}} \mbox{ from some } x \in \mathbb{V}' \mbox{ to $y$} \Big\}.
	\end{equation*}

	In what follows, we present a technical result which we will use in proving our main results of the paper.

	\begin{proposition}\label{prop percolation at smaller probability}
		Let $\beta$ be a $\hat{k}$-cyclic backbend sequence for some $\hat{k} \in \mathbb{N}$. Suppose $0 < p < 1$ and $\theta^{\beta}_{\mathbb{H}}(p) > 0$. Then, there exist $l, r \in \mathbb{N}$ and $0 < \delta <p $ such that
		\begin{equation*}
			\mathbb{P}_{p - \delta} \bigg( \Big| C^{\beta}_{\mathbb{Q}_{l}^2} \Big( B \big([-r, r]^{d-1} \times 0 \big) \Big) \Big| = \infty \bigg) > 0.
		\end{equation*} 
	\end{proposition}

	The proof of this proposition is relegated to Appendix \ref{appendix proof of prop percolation at smaller probability}.

	\subsection{Results on the half-space}\label{subsection results on halfspace}

	Our next theorem establishes a relation between percolation probabilities of a $\tilde{\beta}$-backbend percolation process and a $\beta$-backbend percolation process on the half-space, where $\beta$ is $k$-cyclic and $\tilde{\beta}$ converges from below to $\beta$.

	\begin{theorem}\label{theorem iff condition half space}
		Suppose a backbend sequence $\tilde{\beta}$ converges from below to a $k$-cyclic backbend sequence $\beta$ for some $k \in \mathbb{N}$. Then, for all $p \in [0,1]$, we have $\theta^{\tilde{\beta}}_{\mathbb{H}}(p) > 0$ if and only if $\theta^{\beta}_{\mathbb{H}}(p) > 0$.
	\end{theorem}

	\begin{proof}[\textbf{Proof of Theorem \ref{theorem iff condition half space}}]
		\textbf{\textit{(If part)}} Suppose $\theta^{\beta}_{\mathbb{H}}(p) > 0$. It must be that $p > 0$. Since $\underset{n \rightarrow \infty}{\lim} \hspace{1 mm} \tilde{\beta}_{kn+i} = \beta_i$ for all $i \in \{0, \ldots, k-1\}$, there exists $n^* \in \mathbb{N}$ such that $\tilde{\beta}_{kn+i} = \beta_i$ for all $n \geq n^*$ and all $i \in \{0, \ldots, k-1\}$. Because $\beta$ is $k$-cyclic, by the construction of the $\beta$-backbend percolation process, we have for all $y \in \mathbb{H}$ with $y_d = kn^*$,
		\begin{equation}\label{equation important one halfspace}
			\mathbb{P}_{p} \Big( \big| C^{\beta}_{(\mathbb{H} + y)} (\{y\}) \big| = \infty \Big) = \theta^{\beta}_{\mathbb{H}}(p).\footnotemark
		\end{equation}
		\footnotetext{For $x \in \mathbb{V}$ and $\widehat{\mathbb{V}} \subseteq \mathbb{V}$, we denote by $\widehat{\mathbb{V}} + x$ the set $\{y + x : y \in \widehat{\mathbb{V}}\}$.}Choose $y^* \in \mathbb{H}$ with $y^*_d = kn^*$ such that there exists an oriented path $\pi$ from the origin to $y^*$.\footnote{For instance, $y^*$ can be taken as the vertex that is connected to the origin through the following oriented path: $((0, \ldots, 0), (1, \ldots, 1), (0, \ldots, 0, 2), (1, \ldots, 1, 3), \ldots, (0, \ldots, 0, kn^*))$ if $kn^*$ is even, and $((0, \ldots, 0), (1, \ldots, 1), (0, \ldots, 0, 2), \ldots, (1, \ldots, 1, kn^*))$ if $kn^*$ is odd.} As $p > 0$, the probability of $\pi$ being open is positive.\footnote{Note that since the path $\pi$ is oriented, it will have $kn^*$ edges, and hence $\mathbb{P}_p \big( \pi \mbox{ is open} \big) = p^{kn^*}$.}
		
		Since $\tilde{\beta}_n = \beta_n$ for all $n \geq kn^*$, $y^*_d = kn^*$, and $\pi$ is an oriented path from the origin to $y^*$, it follows that $\pi$ concatenated with any infinite $\beta$-backbend path in $\mathbb{H} + y^*$ from $y^*$ produces an infinite $\tilde{\beta}$-backbend path in $\mathbb{H}$ from the origin. Furthermore, by the construction of $\pi$, the event of $\pi$ being open is independent of the event $\Big\{ \omega \in \Omega : \big| C^{\beta}_{(\mathbb{H} + y^*)} (\{y^*\}) \big| = \infty \Big\}$. Combining these two facts with \eqref{equation important one halfspace}, we have $\theta^{\tilde{\beta}}_{\mathbb{H}}(p) \geq \mathbb{P}_p \big( \pi \mbox{ is open} \big) \times \theta^{\beta}_{\mathbb{H}}(p)$. Because $\mathbb{P}_p \big( \pi \mbox{ is open} \big) > 0$ and $\theta^{\beta}_{\mathbb{H}}(p) > 0$, this implies $\theta^{\tilde{\beta}}_{\mathbb{H}}(p) > 0$. This completes the proof of the ``if'' part of Theorem \ref{theorem iff condition half space}.\medskip
		\\
		\textbf{\textit{(Only-if part)}} Suppose $\theta^{\tilde{\beta}}_{\mathbb{H}}(p) > 0$. By the assumptions on $\tilde{\beta}$ and $\beta$, for all $i \in \{0, \ldots, k-1\}$, we have $\tilde{\beta}_{kn+i} \leq \beta_i$ for all $n \in \mathbb{Z}_+$. This implies that every $\tilde{\beta}$-backbend path in $\mathbb{H}$ from the origin is a $\beta$-backbend path in $\mathbb{H}$ from the origin. Therefore, $\theta^{\beta}_{\mathbb{H}}(p) \geq \theta^{\tilde{\beta}}_{\mathbb{H}}(p)$. Since $\theta^{\tilde{\beta}}_{\mathbb{H}}(p) > 0$, this implies $\theta^{\beta}_{\mathbb{H}}(p) > 0$. This completes the proof of the ``only-if'' part of Theorem \ref{theorem iff condition half space}.
	\end{proof}

	We obtain the following corollary (Corollary \ref{coro same critical percolation at half space}) from Theorem \ref{theorem iff condition half space}. It says that if a backbend sequence $\tilde{\beta}$ converges from below to a $k$-cyclic backbend sequence $\beta$, then on the half-space, the critical probability of $\tilde{\beta}$-backbend percolation process will be equal to that of $\beta$-backbend percolation process.

	\begin{corollary}\label{coro same critical percolation at half space}
		Suppose a backbend sequence $\tilde{\beta}$ converges from below to a $k$-cyclic backbend sequence $\beta$ for some $k \in \mathbb{N}$. Then, $p^{\tilde{\beta}}_c (\mathbb{H}) = p^{\beta}_c (\mathbb{H})$.
	\end{corollary}

	\begin{note}\label{note counter of same critical prob}
		It is worth mentioning that Corollary \ref{coro same critical percolation at half space} does \textit{not} hold if we relax its assumption by requiring that $\tilde{\beta}$ converges (not necessarily from below) to a $k$-cyclic backbend sequence $\beta$ (see Example \ref{example counter for convergent} for details). 
		However, for this case, using the same argument as for the ``if'' part of Theorem \ref{theorem iff condition half space}, we have $p^{\tilde{\beta}}_c (\mathbb{H}) \leq p^{\beta}_c (\mathbb{H})$.
	\end{note}

	\begin{example}\label{example counter for convergent}
		Suppose $d = 3$. 
		Let $\beta$ be a $1$-cyclic backbend sequence such that $\beta_i = 0$ for all $i \in \mathbb{Z}_+$.
		Suppose $\tilde{\beta}$ is a backbend sequence such that $\tilde{\beta}_0 = 0, \tilde{\beta}_1 = 1, \tilde{\beta}_2 = 2, \tilde{\beta}_3 = 3$, and $\tilde{\beta}_i = 0$ for all $i \geq 4$. Clearly, $\tilde{\beta}$ converges to $\beta$. By construction, $p^{\tilde{\beta}}_c (\mathbb{H}) \leq p_c (\mathbb{S}_3)$.
		Since $p_c (\mathbb{S}_3) = 0.21113018(38)$ (see \citet{gliozzi2005random} for details) and $p^{\beta}_c (\mathbb{H}) = 0.2873383(1)$ (see \citet{perlsman2002method} for details), the fact $p^{\tilde{\beta}}_c (\mathbb{H}) \leq p_c (\mathbb{S}_3)$ implies that $p^{\tilde{\beta}}_c (\mathbb{H}) < p^{\beta}_c (\mathbb{H})$.\footnote{Error bars in the last digit or digits are shown by numbers in parentheses. Thus, $0.21113018(38)$ signifies $0.21113018 \pm 0.00000038$.}
	\end{example}

	\begin{remark}\label{remark generalization for 2 dimension}
		In two dimensions, Corollary \ref{coro same critical percolation at half space} holds for the case where $\tilde{\beta}$ converges (not necessarily from below) to a $k$-cyclic backbend sequence $\beta$ satisfying an additional property. See Proposition \ref{prop generalize for 2 dimension} for details.
	\end{remark}

	The next theorem is our main result on half-space which says that if a backbend sequence is $k$-cyclic in the limit from below, then there will be no percolation on the half-space at the critical probability.

	\begin{theorem}\label{theorem no percolation at half space}
		Suppose a backbend sequence $\tilde{\beta}$ is $k$-cyclic in the limit from below for some $k \in \mathbb{N}$. Then, $\theta^{\tilde{\beta}}_{\mathbb{H}}(p^{\tilde{\beta}}_c (\mathbb{H})) = 0$.
	\end{theorem}

	Before we start proving Theorem \ref{theorem no percolation at half space}, we make a remark on the critical probability of any backbend percolation process on the half-space.

	\begin{remark}\label{remark critical prob half space is in 0 and 1}
		For every backbend sequence $\beta$, we have $p^{\beta}_c (\mathbb{H}) < 1$. To see this fix a backbend sequence $\beta$. By the definition of a $\beta$-backbend path, it follows that $p^{\beta}_c (\mathbb{H}) \leq p^{0}_c (\mathbb{H})$. Using a similar argument as in \citet{grimmett2013percolation}, it can be verified that $p^{0}_c (\mathbb{H})$ decreases as $d$ (the number of dimensions) increases. Furthermore, since $2$-dimensional BCC lattice is equivalent to the square lattice, it is known that in two dimensions, $p^{0}_c (\mathbb{H}) < 1$ (see \citet{durrett1984oriented}, \citet{durrett1988lecture}, \citet{balister1993upper}). Combining all these facts, we have $p^{\beta}_c (\mathbb{H}) < 1$.
	\end{remark}

	\begin{proof}[\textbf{Proof of Theorem \ref{theorem no percolation at half space}}]
		Let $\beta$ be the $k$-cyclic backbend sequence such that $\tilde{\beta}$ converges from below to $\beta$. In view of Theorem \ref{theorem iff condition half space}, it is enough to show that $\theta^{\beta}_{\mathbb{H}}(p^{\beta}_c (\mathbb{H})) = 0$.	
		
		Assume for contradiction that $\theta^{\beta}_{\mathbb{H}}(p^{\beta}_c (\mathbb{H})) > 0$. Since $\theta^{\beta}_{\mathbb{H}}(p^{\beta}_c (\mathbb{H})) > 0$, we have $p^{\beta}_c (\mathbb{H}) > 0$. This, together with Remark \ref{remark critical prob half space is in 0 and 1}, implies $0 < p^{\beta}_c (\mathbb{H}) < 1$. Moreover, since $0 < p^{\beta}_c (\mathbb{H}) < 1$ and $\theta^{\beta}_{\mathbb{H}}(p^{\beta}_c (\mathbb{H})) > 0$, by Proposition \ref{prop percolation at smaller probability}, there exist $l, r \in \mathbb{N}$ and $0 < \delta < p^{\beta}_c (\mathbb{H})$ such that
		\begin{equation}\label{equation corollary 1 1}
			\mathbb{P}_{p^{\beta}_c (\mathbb{H}) - \delta} \bigg( \Big| C^{\beta}_{\mathbb{Q}_{l}^2} \Big( B \big([-r, r]^{d-1} \times 0 \big) \Big) \Big| = \infty \bigg) > 0.
		\end{equation}
		Since $\mathbb{Q}_{l}^2 \subseteq \mathbb{H}$, \eqref{equation corollary 1 1} implies
		\begin{equation}\label{equation corollary 1 2}
			\mathbb{P}_{p^{\beta}_c (\mathbb{H}) - \delta} \bigg( \Big| C^{\beta}_{\mathbb{H}} \Big( B \big([-r, r]^{d-1} \times 0 \big) \Big) \Big| = \infty \bigg) > 0.
		\end{equation}
		Because $B \big([-r, r]^{d-1} \times 0 \big)$ is a finite set, \eqref{equation corollary 1 2} implies that there exists some $x \in B \big([-r, r]^{d-1} \times 0 \big)$ such that $\mathbb{P}_{p^{\beta}_c (\mathbb{H}) - \delta} \Big( \big| C^{\beta}_{\mathbb{H}} (\{x\}) \big| = \infty \Big) > 0$. Furthermore, by the construction of the $\beta$-backbend percolation process, we have $\mathbb{P}_{p^{\beta}_c (\mathbb{H}) - \delta} \Big( \big| C^{\beta}_{\mathbb{H}} \big| = \infty \Big) = \mathbb{P}_{p^{\beta}_c (\mathbb{H}) - \delta} \Big( \big| C^{\beta}_{\mathbb{H}} (\{x\}) \big| = \infty \Big)$, and hence 
		\begin{equation}\label{equation corollary 1 3}
			\mathbb{P}_{p^{\beta}_c (\mathbb{H}) - \delta} \Big( \big| C^{\beta}_{\mathbb{H}} \big| = \infty \Big) > 0.
		\end{equation}
		However, since $\delta >0$, \eqref{equation corollary 1 3} contradicts the definition of $p^{\beta}_c (\mathbb{H})$. This completes the proof of Theorem \ref{theorem no percolation at half space}.
	\end{proof}

	\subsection{Results on the full-space}\label{subsection result on full-space}

	In this subsection, we consider the full-space and a particular class of backbend sequences. These backbend sequences have the property that they converge from below to a $k$-cyclic backbend sequence $\beta$ satisfying the property that $\beta_{1} \leq \beta_0 + 1, \ldots, \beta_{k-1} \leq \beta_{k-2} + 1, \mbox{ and } \beta_{0} \leq \beta_{k-1} + 1$ (that is, $\beta_{i+1} \leq \beta_i + 1$ for all $i \in \{0, \ldots, k-1\}$). For such backbend sequences, the first part of Theorem \ref{theorem on full space} says that the critical probability on the full-space will be the same as that on the half-space, and the second part of the theorem says that there will be no percolation on the full-space at the critical probability.

	\begin{theorem}\label{theorem on full space}
		Let $k \in \mathbb{N}$ and suppose that a backbend sequence $\tilde{\beta}$ converges from below to a $k$-cyclic backbend sequence $\beta$ satisfying $\beta_{i+1} \leq \beta_i + 1$ for all $i \in \{0, \ldots, k-1\}$. Then,
		\begin{enumerate}[(i)]
			\item\label{item same critical percolation at half full} $p^{\tilde{\beta}}_c (\mathbb{V}) = p^{\tilde{\beta}}_c (\mathbb{H})$, and
			
			\item\label{item no percolation full space} $\theta^{\tilde{\beta}}_{\mathbb{V}}(p^{\tilde{\beta}}_c (\mathbb{V})) = 0$.
		\end{enumerate}
	\end{theorem}
	
	\begin{proof}[\textbf{Proof of Theorem \ref{theorem on full space}}]
		We first prove a claim which we will use in the proof of the theorem.

		\begin{claim}\label{claim equalilty of percolation probablity cycle}
			For all $p \in [0,1]$, $\theta^{\beta}_{\mathbb{V}}(p) > 0$ implies $\theta^{\beta}_{\mathbb{H}}(p) > 0$.
		\end{claim}

		\begin{claimproof}[\textbf{Proof of Claim \ref{claim equalilty of percolation probablity cycle}}]
			Suppose $\theta^{\beta}_{\mathbb{V}}(p) > 0$. It must be that $p > 0$. Let $n^* \in k\mathbb{N}$ be such that $n^* \geq \max \{\beta_0, \ldots, \beta_{k-1}\}$.\footnote{For $k \in \mathbb{N}$, we denote by $k\mathbb{N}$ the set $\{kn : n \in \mathbb{N}\}$.} Thus, $n^*$ is such that any $\beta$-backbend path starting from a vertex $y$ with $y_d = n^*$ will always be in the half-space $\mathbb{H}$. Because $\beta$ is $k$-cyclic, by the construction of the $\beta$-backbend percolation process, we have for all $y \in \mathbb{H}$ with $y_d = n^*$,
			\begin{equation}\label{equation important one}
				\mathbb{P}_{p} \Big( \big| C^{\beta}_{\mathbb{V}} (\{y\}) \big| = \infty \Big) = \theta^{\beta}_{\mathbb{V}}(p).
			\end{equation}
			Choose $y^* \in \mathbb{H}$ with $y^*_d = n^*$ such that there exists an oriented path $\pi$ from the origin to $y^*$. As $p > 0$, the probability of $\pi$ being open is positive.
			
			For each infinite $\beta$-backbend path $\hat{\pi}$ in $\mathbb{V}$ from $y^*$, we construct the path $\hat{\pi}^*$ by distinguishing the following cases.
			\begin{enumerate}[(i)]
				\item Suppose $\pi$ and $\hat{\pi}$ have no common vertex other than $y^*$. Then, $\hat{\pi}^*$ is obtained by concatenating the paths $\pi$ and $\hat{\pi}$.
				
				\item Suppose $\pi$ and $\hat{\pi}$ have common vertices other than $y^*$. Let $z^*$ be the first vertex of $\pi$ such that $z^* \in \hat{\pi}$. Then $\hat{\pi}^*$ is obtained by concatenating $\pi_s$ and $\hat{\pi}_s$, where $\pi_s$ is the sub-path of $\pi$ from the origin to $z^*$ and $\hat{\pi}_s$ is the sub-path of $\hat{\pi}$ from $z^*$.
			\end{enumerate}
			To see that $\hat{\pi}^*$ is indeed a path, observe that $\hat{\pi}^*$ is self-avoiding by construction. We claim that for each infinite $\beta$-backbend path $\hat{\pi}$ in $\mathbb{V}$ from $y^*$, the path $\hat{\pi}^*$ constructed as above is an infinite $\beta$-backbend path in $\mathbb{H}$ from the origin. Clearly, $\hat{\pi}^*$ is an infinite path from the origin. Moreover, as we have mentioned earlier, by the choice of $n^*$, the path $\hat{\pi}$ is in $\mathbb{H}$, and hence the path $\hat{\pi}^*$ is in $\mathbb{H}$. We proceed to show that $\hat{\pi}^*$ is a $\beta$-backbend path. 
			
			Assume for contradiction that $\hat{\pi}^*$ is not a $\beta$-backbend path. Since $\hat{\pi}^*$ is not a $\beta$-backbend path, there exists a vertex $x^*$ in $\hat{\pi}^*$ such that $x^*_d < h^* - \beta_{h^*}$, where $h^*$ is the record level attained by the path $\hat{\pi}^*$ till $x^*$. Since $\pi$ is an oriented path, by the construction of $\hat{\pi}^*$, it follows that $x^*$ must be in $\hat{\pi}$. Let $\hat{h}$ be the record level attained by the path $\hat{\pi}$ till $x^*$. Because $\hat{\pi}$ is a $\beta$-backbend path, it must be that $x^*_d \geq \hat{h} - \beta_{\hat{h}}$. This, together with the fact $x^*_d < h^* - \beta_{h^*}$, yields
			\begin{equation}\label{equation concate beta path}
				h^* - \beta_{h^*} > \hat{h} - \beta_{\hat{h}}.
			\end{equation}
			Since $\pi$ is an oriented path to $y^*$ and $\hat{\pi}$ is a path from $y^*$, by the construction of $\hat{\pi}^*$, we have $\hat{h} \geq h^*$. The assumptions on $\beta$ imply that $\beta_l - \beta_m \leq l-m$ for all $l \geq m$. Since $\hat{h} \geq h^*$, this yields $\beta_{\hat{h}} - \beta_{h^*} \leq \hat{h} - h^*$, a contradiction to \eqref{equation concate beta path}. So, it must be that $\hat{\pi}^*$ is a $\beta$-backbend path.
			
			Since for each infinite $\beta$-backbend path $\hat{\pi}$ in $\mathbb{V}$ from $y^*$, $\hat{\pi}^*$ is an infinite $\beta$-backbend path in $\mathbb{H}$ from the origin, by the construction of $\hat{\pi}^*$, we have
			\begin{equation*}
				\theta^{\beta}_{\mathbb{H}}(p) \geq \mathbb{P}_p \Big( \pi \mbox{ is open and } \big| C^{\beta}_{\mathbb{V}} (\{y^*\}) \big| = \infty \Big),
			\end{equation*}
			and hence, by FKG inequality (see \citet{grimmett2013percolation} for details),
			\begin{equation*}
				\theta^{\beta}_{\mathbb{H}}(p) \geq \mathbb{P}_p \big( \pi \mbox{ is open} \big) \times \mathbb{P}_p \Big( \big| C^{\beta}_{\mathbb{V}} (\{y^*\}) \big| = \infty \Big).
			\end{equation*}
			By \eqref{equation important one}, this gives $\theta^{\beta}_{\mathbb{H}}(p) \geq \mathbb{P}_p \big( \pi \mbox{ is open} \big) \times \theta^{\beta}_{\mathbb{V}}(p)$.
			Since $\mathbb{P}_p \big( \pi \mbox{ is open} \big) > 0$ and $\theta^{\beta}_{\mathbb{V}}(p) > 0$, we have $\theta^{\beta}_{\mathbb{H}}(p) > 0$. This completes the proof of Claim \ref{claim equalilty of percolation probablity cycle}.
		\end{claimproof}

		Now, we complete the proof of Theorem \ref{theorem on full space}. In view of Theorem \ref{theorem no percolation at half space}, it is enough to show that for all $p \in [0,1]$, we have $\theta^{\tilde{\beta}}_{\mathbb{V}}(p) > 0$ if and only if $\theta^{\tilde{\beta}}_{\mathbb{H}}(p) > 0$.

		Since $\mathbb{H} \subset \mathbb{V}$, it follows that for all $p \in [0,1]$, $\theta^{\tilde{\beta}}_{\mathbb{H}}(p) > 0$ implies $\theta^{\tilde{\beta}}_{\mathbb{V}}(p) > 0$. We proceed to show that for all $p \in [0,1]$, $\theta^{\tilde{\beta}}_{\mathbb{V}}(p) > 0$ implies $\theta^{\tilde{\beta}}_{\mathbb{H}}(p) > 0$. Suppose $\theta^{\tilde{\beta}}_{\mathbb{V}}(p) > 0$ for some $p \in [0,1]$. By the assumptions on $\tilde{\beta}$ and $\beta$, it follows that every $\tilde{\beta}$-backbend path in $\mathbb{V}$ from the origin is a $\beta$-backbend path in $\mathbb{V}$ from the origin. Therefore, $\theta^{\tilde{\beta}}_{\mathbb{V}}(p) > 0$ implies $\theta^{\beta}_{\mathbb{V}}(p) > 0$. This, together with the assumptions on $\beta$ and Claim \ref{claim equalilty of percolation probablity cycle}, yields $\theta^{\beta}_{\mathbb{H}}(p) > 0$. Moreover, since $\tilde{\beta}$ converges from below to $\beta$ and $\theta^{\beta}_{\mathbb{H}}(p) > 0$, by Theorem \ref{theorem iff condition half space}, we have $\theta^{\tilde{\beta}}_{\mathbb{H}}(p) > 0$. This completes the proof of Theorem \ref{theorem on full space}.
	\end{proof}

	\subsection{Results on half-slabs}\label{subsection result on half-slabs}

	Throughout this subsection, we assume that $d \geq 3$. Our next theorem (Theorem \ref{theorem iff condition half slabs}) establishes a relation between percolation probabilities of a $\tilde{\beta}$-backbend percolation process and a $\beta$-backbend percolation process on every $e$-dimensional half-slab, where $\beta$ is $k$-cyclic and $\tilde{\beta}$ converges from below to $\beta$. The proof of the ``only-if'' part of the theorem is similar to that of Theorem \ref{theorem iff condition half space}, and the proof of the ``if'' part of the theorem follows from that of Theorem \ref{theorem iff condition half space} by additionally imposing the following requirements: (i) $n^* \in 2\mathbb{N}$, (ii) $y^* = (0, \ldots, 0, kn^*)$, and (iii) $\pi$ is in $\mathbb{Q}_l^e$. It can be verified that these additional requirements do not affect the logic of the proof of the ``if'' part of Theorem \ref{theorem iff condition half space}.

	\begin{theorem}\label{theorem iff condition half slabs}
		Suppose a backbend sequence $\tilde{\beta}$ converges from below to a $k$-cyclic backbend sequence $\beta$ for some $k \in \mathbb{N}$. Then, for all $l \in \mathbb{N}$, all $e \in \{2, \ldots, d-1\}$, and all $p \in [0,1]$, we have $\theta^{\tilde{\beta}}_{\mathbb{Q}_l^e}(p) > 0$ if and only if $\theta^{\beta}_{\mathbb{Q}_l^e}(p) > 0$.
	\end{theorem}
	
	We obtain the following corollary (Corollary \ref{coro same critical percolation at half slabs}) from Theorem \ref{theorem iff condition half slabs}. It says that if a backbend sequence $\tilde{\beta}$ converges from below to a $k$-cyclic backbend sequence $\beta$, then on every $e$-dimensional half-slab, the critical probability of $\tilde{\beta}$-backbend percolation process will be equal to that of $\beta$-backbend percolation process.

	\begin{corollary}\label{coro same critical percolation at half slabs}
		Suppose a backbend sequence $\tilde{\beta}$ converges from below to a $k$-cyclic backbend sequence $\beta$ for some $k \in \mathbb{N}$. Then, for all $l \in \mathbb{N}$ and all $e \in \{2, \ldots, d-1\}$, we have $p^{\tilde{\beta}}_c (\mathbb{Q}_l^e) = p^{\beta}_c (\mathbb{Q}_l^e)$.
	\end{corollary}

	Our next theorem says that if a backbend sequence is $k$-cyclic in the limit from below, then for all $e \in \{2, \ldots, d -1\}$, the limit of the critical probability on the $e$-dimensional half-slab $\mathbb{Q}_l^e$ as $l$ goes to infinity will be equal to the critical probability on the half-space.

	\begin{theorem}\label{theorem limiting value of half slabs}
		Suppose a backbend sequence $\tilde{\beta}$ is $k$-cyclic in the limit from below for some $k \in \mathbb{N}$. Then, for all $e \in \{2, \ldots, d-1\}$, we have $\underset{l \rightarrow \infty}{\lim} \hspace{1 mm} p^{\tilde{\beta}}_c (\mathbb{Q}_l^e) = p^{\tilde{\beta}}_c (\mathbb{H})$.			
	\end{theorem}

	\begin{proof}[\textbf{Proof of Theorem \ref{theorem limiting value of half slabs}}]
		Fix an arbitrary $e \in \{2, \ldots, d-1\}$. Let $\beta$ be the $k$-cyclic backbend sequence such that $\tilde{\beta}$ converges from below to $\beta$. In view of Corollary \ref{coro same critical percolation at half space} and Corollary \ref{coro same critical percolation at half slabs}, it is enough to show that $\underset{l \rightarrow \infty}{\lim} \hspace{1 mm} p^{\beta}_c (\mathbb{Q}_l^e) = p^{\beta}_c (\mathbb{H})$.
		
		Since $\mathbb{Q}_l^e \subset \mathbb{Q}_{l+1}^e \subset \mathbb{H}$ for all $l \in \mathbb{N}$, we have $\underset{l \rightarrow \infty}{\lim} \hspace{1 mm} p^{\beta}_c (\mathbb{Q}_l^e) \geq p^{\beta}_c (\mathbb{H})$. Assume for contradiction that $\underset{l \rightarrow \infty}{\lim} \hspace{1 mm} p^{\beta}_c (\mathbb{Q}_l^e) > p^{\beta}_c (\mathbb{H})$. Let $p \in [0,1]$ be such that $p^{\beta}_c (\mathbb{H}) < p < \underset{l \rightarrow \infty}{\lim} \hspace{1 mm} p^{\beta}_c (\mathbb{Q}_l^e)$. It must be that $0 < p < 1$ and $\theta^{\beta}_{\mathbb{H}}(p) > 0$. Since $0 < p < 1$ and $\theta^{\beta}_{\mathbb{H}}(p) > 0$, by Proposition \ref{prop percolation at smaller probability}, there exist $l^*, r^* \in \mathbb{N}$ and $0 < \delta < p$ such that
		\begin{equation*}
			\mathbb{P}_{p - \delta} \bigg( \Big| C^{\beta}_{\mathbb{Q}_{l^*}^2} \Big( B \big([-r^*, r^*]^{d-1} \times 0 \big) \Big) \Big| = \infty \bigg) > 0,
		\end{equation*}
		and hence
		\begin{equation}\label{equation limiting 1}
			\mathbb{P}_p \bigg( \Big| C^{\beta}_{\mathbb{Q}_{l^*}^2} \Big( B \big([-r^*, r^*]^{d-1} \times 0 \big) \Big) \Big| = \infty \bigg) > 0.
		\end{equation}
		Because $B \big([-r^*, r^*]^{d-1} \times 0 \big)$ is a finite set, \eqref{equation limiting 1} implies that there exists some $x \in B \big([-r^*, r^*]^{d-1} \times 0 \big)$ such that $\mathbb{P}_p \Big( \big| C^{\beta}_{(x + \mathbb{Q}_{l^*+r^*}^2)} (\{x\}) \big| = \infty \Big) > 0$. Furthermore, by the construction of the $\beta$-backbend percolation process, we have $\mathbb{P}_p \Big( \big| C^{\beta}_{\mathbb{Q}_{l^*+r^*}^2} \big| = \infty \Big) = \mathbb{P}_p \Big( \big| C^{\beta}_{(x + \mathbb{Q}_{l^*+r^*}^2)} (\{x\}) \big| = \infty \Big)$, which implies $\mathbb{P}_p \Big( \big| C^{\beta}_{\mathbb{Q}_{l^*+r^*}^2} \big| = \infty \Big) > 0$. Since $\mathbb{Q}_{l^*+r^*}^2 \subseteq \mathbb{Q}_{l^*+r^*}^e$, this gives
		\begin{equation}\label{equation limiting 2}
			\mathbb{P}_p \Big( \big| C^{\beta}_{\mathbb{Q}_{l^*+r^*}^e} \big| = \infty \Big) > 0.
		\end{equation}
		However, since $p < \underset{l \rightarrow \infty}{\lim} \hspace{1 mm} p^{\beta}_c (\mathbb{Q}_l^e)$ and $p^{\beta}_c (\mathbb{Q}_l^e)$ decreases as $l$ increases, it must be that $p < p^{\beta}_c (\mathbb{Q}_{l^*+r^*}^e)$, a contradiction to \eqref{equation limiting 2}. This completes the proof of Theorem \ref{theorem limiting value of half slabs}.
	\end{proof}

	\begin{note}\label{note counter for not limiting}
		As we have noted (Note \ref{note counter of same critical prob}) for Corollary \ref{coro same critical percolation at half space}, Theorem \ref{theorem limiting value of half slabs} also does \textit{not} hold in general.
		More precisely, as Example \ref{example counter for limiting} shows, if we relax the assumption of Theorem \ref{theorem limiting value of half slabs} by requiring that $\tilde{\beta}$ converges (not necessarily from below) to a $k$-cyclic backbend sequence $\beta$, then Theorem \ref{theorem limiting value of half slabs} does not hold on two-dimensional half-slabs when $\beta$ satisfies the condition introduced in Subsection \ref{subsection result on full-space} (that is, $\beta_{i+1} \leq \beta_i + 1$ for all $i \in \{0, \ldots, k-1\}$).
	\end{note}

	We use the following proposition in Example \ref{example counter for limiting}. It provides an upper bound and a lower bound of the critical probability of a $\tilde{\beta}$-backbend percolation process on a two-dimensional half-slab when $\tilde{\beta}$ converges to a $k$-cyclic backbend sequence $\beta$ satisfying $\beta_{i+1} \leq \beta_i + 1$ for all $i \in \{0, \ldots, k-1\}$.

	\begin{proposition}\label{prop bounds two dimensional half slab}
		Let $k \in \mathbb{N}$ and suppose that a backbend sequence $\tilde{\beta}$ converges to a $k$-cyclic backbend sequence $\beta$ satisfying $\beta_{i+1} \leq \beta_i + 1$ for all $i \in \{0, \ldots, k-1\}$. Then, for all $l \in \mathbb{N}$, we have $p^{\beta}_c (\mathbb{Q}_{2l}^2) \leq p^{\tilde{\beta}}_c (\mathbb{Q}_l^2) \leq p^{\beta}_c (\mathbb{Q}_l^2)$.
	\end{proposition}

	The proof of this proposition is relegated to Appendix \ref{appendix proof of prop bounds two dimensional half slab}.
	
	We are now ready to present our counter example.

	\begin{example}\label{example counter for limiting}
		Suppose $d = 3$.
		Consider the backbend sequences $\beta$ and $\tilde{\beta}$ given in Example \ref{example counter for convergent}, where $\beta$ is $1$-cyclic and $\tilde{\beta}$ converges to $\beta$. We have already shown that $p^{\tilde{\beta}}_c (\mathbb{H}) < p^{\beta}_c (\mathbb{H})$ in Example \ref{example counter for convergent}.
		By the assumption on $\beta$ and $\tilde{\beta}$, it follows from Proposition \ref{prop bounds two dimensional half slab} that $\underset{l \rightarrow \infty}{\lim} \hspace{1 mm} p^{\tilde{\beta}}_c (\mathbb{Q}_l^2) = \underset{l \rightarrow \infty}{\lim} \hspace{1 mm} p^{\beta}_c (\mathbb{Q}_l^2)$.
		Furthermore, by the assumption on $\beta$, it follows from Theorem \ref{theorem limiting value of half slabs} that $\underset{l \rightarrow \infty}{\lim} \hspace{1 mm} p^{\beta}_c (\mathbb{Q}_l^2) = p^{\beta}_c (\mathbb{H})$.
		The facts $p^{\tilde{\beta}}_c (\mathbb{H}) < p^{\beta}_c (\mathbb{H})$, $\underset{l \rightarrow \infty}{\lim} \hspace{1 mm} p^{\tilde{\beta}}_c (\mathbb{Q}_l^2) = \underset{l \rightarrow \infty}{\lim} \hspace{1 mm} p^{\beta}_c (\mathbb{Q}_l^2)$, and $\underset{l \rightarrow \infty}{\lim} \hspace{1 mm} p^{\beta}_c (\mathbb{Q}_l^2) = p^{\beta}_c (\mathbb{H})$ together imply $\underset{l \rightarrow \infty}{\lim} \hspace{1 mm} p^{\tilde{\beta}}_c (\mathbb{Q}_l^2) > p^{\tilde{\beta}}_c (\mathbb{H})$.
	\end{example}
	
	In the following, we present a result on the half-space in two dimensions (as we have mentioned in Remark \ref{remark generalization for 2 dimension}). The proof of this result follows by using similar logic as for the proof of Proposition \ref{prop bounds two dimensional half slab}.
	
	\begin{proposition}\label{prop generalize for 2 dimension}
		Suppose $d = 2$. Let $k \in \mathbb{N}$ and suppose that a backbend sequence $\tilde{\beta}$ converges to a $k$-cyclic backbend sequence $\beta$ satisfying $\beta_{i+1} \leq \beta_i + 1$ for all $i \in \{0, \ldots, k-1\}$. Then, $p^{\tilde{\beta}}_c (\mathbb{H}) = p^{\beta}_c (\mathbb{H})$.
	\end{proposition}

	\appendixtocon
	\appendixtitletocon

	\renewcommand{\theequation}{\Alph{section}.\arabic{equation}}
	\renewcommand{\thetable}{\Alph{section}.\arabic{table}}
	\renewcommand{\thefigure}{\Alph{section}.\arabic{figure}}

	\begin{appendices}

		\section{Proof of Proposition \ref{prop percolation at smaller probability}}\label{appendix proof of prop percolation at smaller probability}

		\setcounter{table}{0}
		\setcounter{figure}{0}
		\setcounter{equation}{0}

		We first introduce some notations and make a remark to facilitate the proof. 
		
		For all $l, t \in \mathbb{Z}_+ \cup \{\infty\}$, all $u \in \{-1, 1\}^{d-1}$, and all $v \in \{-1, 1\}^{d-2}$, let us define
		\begin{equation*}
			\begin{aligned}
				& T(l,t) := \Big\{ x \in B \big([-l, l]^{d-1} \times [0, t] \big) : x_d = t \Big\},\\
				& F(l,t) := \Big\{ x \in B \big([-l, l]^{d-1} \times [0, t] \big) : |x_i| = l \mbox{ for some } i \in \{1, \ldots, d-1\} \Big\}, \\
				& F_{(d-1)^{+}}(l,t) := \Big\{ x \in F(l,t) : x_{d-1} = l \Big\}, \\
				& T^{u}(l,t) := \Big\{ x \in T(l,t) : 0 \leq x_iu_i \leq l \mbox{ for all } i \in \{1, \ldots, d-1\} \Big\}, \mbox{ and}\\
				& F^{v}_{(d-1)^{+}}(l,t) := \Big\{ x \in F_{(d-1)^{+}}(l,t) : 0 \leq x_iv_i \leq l \mbox{ for all } i \in \{1, \ldots, d-2\} \Big\}.
			\end{aligned}
		\end{equation*}

		The following remark follows from basic probability.

		\begin{remark}\label{remark probability inequality}
			Let $\alpha > 0$ and let $A_1, A_2, A_3$ be events such that $A_2$ and $A_3$ are disjoint with $\mathbb{P}_p (A_2) > 0$ and $\mathbb{P}_p (A_3) > 0$. 
			\begin{enumerate}[(a)]
				\item Suppose $\mathbb{P}_p(A_1 \mid A_2) < \alpha$ and $\mathbb{P}_p(A_1 \mid A_3) < \alpha$. Then, $\mathbb{P}_p \big( A_1 \bigm| (A_2 \cup A_3) \big) < \alpha$.
				
				\item Suppose $\mathbb{P}_p(A_1 \mid A_2) > \alpha$ and $\mathbb{P}_p(A_1 \mid A_3) > \alpha$. Then, $\mathbb{P}_p \big( A_1 \bigm| (A_2 \cup A_3) \big) > \alpha$.
				
				\item Suppose $\mathbb{P}_p(A_1 \mid A_2) = \alpha$ and $\mathbb{P}_p(A_1 \mid A_3) = \alpha$. Then, $\mathbb{P}_p \big( A_1 \bigm| (A_2 \cup A_3) \big) = \alpha$.
			\end{enumerate}
		\end{remark}

		\subsection{The proof}

		Let $\eta^* \in (0,1)$ be sufficiently small. Choose $i^* \in \mathbb{N}$ with $i^* > 10$, and $\epsilon^* \in (0,1)$ such that $(1 - \epsilon^*)^{2i^*} > 1 - \eta^*$. We prove Proposition \ref{prop percolation at smaller probability} by distinguishing the following two cases.\medskip
		\\
		\noindent\textbf{\textsc{Case} 1}: Suppose $\theta^{\beta}_{\mathbb{S}_t}(p) = 0$ for all $t \in \mathbb{N}$.			
		
		\noindent Since $\theta^{\beta}_{\mathbb{H}}(p) > 0$, by a standard argument (see \citet{liggett2012interacting}, Theorem 1.10(d) in Chapter VI for details), there exists $r^* \in 2\hat{k}\mathbb{N}$ with $2r^* \geq \max \{\beta_0, \ldots, \beta_{\hat{k}-1}\}$ such that
		\begin{equation}\label{equation from lemma almost 1 prob case 1}
			\mathbb{P}_{p} \Bigg( \Big| C^{\beta}_{\mathbb{H}} \Big( B \big([-r^*, r^*]^{d-1} \times 0 \big) \Big) \Big| = \infty \Bigg) > 1 - \frac{1}{2^{2^{d-1}}} \bigg( \frac{\epsilon^*}{5} \bigg)^{d2^{d-1}}.
		\end{equation}
		For ease of presentation, let us denote $B \big([-r^*, r^*]^{d-1} \times 0 \big)$ by $D^*$. We define the following terms involving $\epsilon^*$ and $r^*$ for our next lemma. 
		\begin{enumerate}[$\bullet$]	
			\item Let $\alpha^* : = \min \left\{
			\begin{aligned}
				& \mathbb{P}_p \Big( B \big([-r^*, r^*]^{d-1} \times [2r^*, 2r^* + 2\hat{k}] \big) \subseteq C^{0}_{B \big([-r^*, r^*]^{d-1} \times [0, 2r^* + 2\hat{k}] \big)} \Big),\\
				& \mathbb{P}_p \Big( B \big([0, 2r^*] \times [-r^*, r^*]^{d-2} \times [2r^*, 2r^* + 2\hat{k}] \big) \subseteq C^{0}_{B \big([0, 2r^*] \times [-r^*, r^*]^{d-2} \times [0, 2r^* + 2\hat{k}] \big)} \Big)
			\end{aligned}
			\right\}$.\footnote{By $C^{0}_{\widehat{\mathbb{V}}}$, we denote $C^{\beta}_{\widehat{\mathbb{V}}}$ where $\beta_i = 0$ for all $i \in \mathbb{Z}_+$.} Clearly, $\alpha^* > 0$.
									
			\item Let $m^* \in \mathbb{N}$ be such that $(1 - \alpha^*)^{m^*} < \frac{\epsilon^*}{5}$. The implication of $m^*$ is that if one performs at least $m^*$ independent trials with the probability of success in each trial being $\alpha^*$, then the probability that there is at least one success  exceeds $1 - \frac{\epsilon^*}{5}$.
			
			\item Let $n^* \in \mathbb{N}$ be such that $n^* \geq m^*(8r^* + 8\hat{k} + 1)^d$. The purpose of $n^*$ is to ensure that in any subset of $\mathbb{V}$ having size $n^*$ or larger, there are at least $m^*$ vertices such the $L^{\infty}$-distance between any two of them is at least $(4r^* + 4\hat{k} + 1)$.\footnote{$L^{\infty}(x,y) := \underset{i}{\max} \hspace{1 mm} \big\{\vert x_i - y_i \vert \big\}$.}
		\end{enumerate}

		\begin{lemma}\label{lemma enough vertices case 1}
			There exist $l^*, t^* \in 2\hat{k}\mathbb{N}$ with $l^* \geq r^*$ such that
			\begin{enumerate}[(i)]
				\item $\mathbb{P}_{p} \bigg( \Big| C^{\beta}_{B \big([-l^*, l^*]^{d-1} \times [0, t^*] \big)}(D^*) \cap T^{u}(l^*, t^*) \Big| \geq n^* \bigg) > 1 - \frac{\epsilon^*}{5}$ for all $u \in \{-1, 1\}^{d-1}$, and
				
				\item $\mathbb{P}_{p} \bigg( \Big| C^{\beta}_{B \big([-l^*, l^*]^{d-1} \times [0, t^*] \big)}(D^*) \cap F^{v}_{(d-1)^{+}}(l^*, t^*) \Big| \geq n^* \bigg) > 1 - \frac{\epsilon^*}{5}$ for all $v \in \{-1, 1\}^{d-2}$.
			\end{enumerate}
		\end{lemma}

		The proof of this lemma is relegated to Appendix \ref{appendix Proof of lemma enough vertices case 1}.
		
		Suppose $l^*$ and $t^*$ are such that the statement of Lemma \ref{lemma enough vertices case 1} holds. Let $s^* := t^* + 2r^* + 2\hat{k}$. Since $2r^* \geq \max \{\beta_0, \ldots, \beta_{\hat{k}-1}\}$ and $r^*, t^* \in 2\hat{k}\mathbb{N}$, by the construction of $s^*$, we have $s^* \in 2\hat{k}\mathbb{N}$ and $s^* > \max \{\beta_0, \ldots, \beta_{\hat{k}-1}\}$.

		\begin{lemma}\label{lemma coupling case 1}
			There exists $0 < \delta < p$ such that for all $x \in B \big([-l^*,l^*]^{d-2} \times 0 \times 0 \big)$, we have
			\begin{equation*}
				\mathbb{P}_{p-\delta} 
				\left(
				\begin{aligned}
					& \big( D^* + z \big) \subseteq C^{\beta}_{B \big([-2l^*, 2l^*]^{d-2} \times [-l^*, 3l^*] \times [0, z_d] \big)}(D^* + x)\\
					& \mbox{ for some } z \in B \big([-l^*, l^*]^{d-2} \times [l^*, 2l^*] \times [s^*, 2s^*] \big) \mbox{ with } z_d \in 2\hat{k}\mathbb{N} 
				\end{aligned}
				\right) > 1 - \epsilon^*.
			\end{equation*}
		\end{lemma}

		The proof of this lemma is relegated to Appendix \ref{appendix Proof of lemma coupling case 1}.
		
		Suppose $\delta^*$ is such that the statement of Lemma \ref{lemma coupling case 1} holds.

		\begin{lemma}\label{lemma transition case 1}
			For all $x \in B \big([-l^*,l^*]^{d-2} \times [-2l^*, 2l^*] \times [0, 2s^*] \big)$ with $x_d \in 2\hat{k}\mathbb{Z}_+$, we have
			\begin{equation*}
				\mathbb{P}_{p-\delta^*} 
				\left(
				\begin{aligned}
					& \big( D^* + z \big) \subseteq C^{\beta}_{B \big([-2l^*, 2l^*]^{d-2} \times [-3l^*, 4l^*] \times [x_d, z_d] \big)}(D^* + x)\\
					& \mbox{ for some } z \in B \big([-l^*, l^*]^{d-2} \times [-l^*, 3l^*] \times [2s^*, 4s^*] \big) \mbox{ with } z_d \in 2\hat{k}\mathbb{N} 
				\end{aligned}
				\right) > (1 - \epsilon^*)^2.
			\end{equation*}
		\end{lemma}

		The proof of this lemma is relegated to Appendix \ref{appendix Proof of lemma transition case 1}.
		
		For all $i \in \mathbb{N}$, we denote the set of vertices $B \big([-l^*, l^*]^{d-2} \times [(i-2)l^*, (i+2)l^*] \times [2is^*, 2(i+1)s^*] \big)$ by $B^{+}_i$. For all $i \in \mathbb{N}$ and all $z \in B^{+}_i$, let us define
		\begin{equation*}
			\mathcal{R}^{+}(z) := \left\{ y \in \mathbb{H} : 
			\begin{aligned}
				& \hspace{1 mm} -2l^* \leq y_j \leq 2l^* \mbox{ for all } j \in \{1, \ldots, d-2\},\\
				& \hspace{1 mm} -5l^* + \frac{l^*}{2s^*} y_d \leq y_{d-1} \leq 5l^* + \frac{l^*}{2s^*} y_d, \mbox{ and } 0 \leq y_d \leq z_d
			\end{aligned} 
			\right\}.
		\end{equation*}
		For all $i \in \mathbb{N}$ and all $x \in B \big([-l^*,l^*]^{d-2} \times [-2l^*, 2l^*] \times [0, 2s^*] \big)$ with $x_d \in 2\hat{k}\mathbb{Z}_+$, let us define the following event:
		\begin{equation*}
			G_i^{+}(x) := \left\{ \omega \in \Omega : \big( D^* + z \big) \subseteq C^{\beta}_{\mathcal{R}^{+}(z)}(D^* + x) \mbox{ for some } z \in B^{+}_i \mbox{ with } z_d \in 2\hat{k}\mathbb{N} \right\}.
		\end{equation*}

		\begin{lemma}\label{lemma iteration case 1}
			For all $i \in \mathbb{N}$ and all $x \in B \big([-l^*,l^*]^{d-2} \times [-2l^*, 2l^*] \times [0, 2s^*] \big)$ with $x_d \in 2\hat{k}\mathbb{Z}_+$, we have $\mathbb{P}_{p-\delta^*} \big( G_i^{+}(x) \big) > (1 - \epsilon^*)^{2i}$.
		\end{lemma}
		
		The proof of this lemma is relegated to Appendix \ref{appendix Proof of lemma iteration case 1}.

		For all $i \in \mathbb{N}$, we denote the set of vertices $B \big([-l^*, l^*]^{d-2} \times [(-i-2)l^*, (-i+2)l^*] \times [2is^*, 2(i+1)s^*] \big)$ by $B^{-}_i$. For all $i \in \mathbb{N}$ and all $z \in B^{-}_i$, let us define
		\begin{equation*}
			\mathcal{R}^{-}(z) := \left\{ y \in \mathbb{H} :
			\begin{aligned}
				& \hspace{1 mm} -2l^* \leq y_j \leq 2l^* \mbox{ for all } j \in \{1, \ldots, d-2\},\\
				& \hspace{1 mm} -5l^* - \frac{l^*}{2s^*} y_d \leq y_{d-1} \leq 5l^* - \frac{l^*}{2s^*} y_d, \mbox{ and } 0 \leq y_d \leq z_d
			\end{aligned} 
			\right\}.
		\end{equation*}
		For all $i \in \mathbb{N}$ and all $x \in B \big([-l^*,l^*]^{d-2} \times [-2l^*, 2l^*] \times [0, 2s^*] \big)$ with $x_d \in 2\hat{k}\mathbb{Z}_+$, let us define the following event:
		\begin{equation*}
			G_i^{-}(x) := \left\{ \omega \in \Omega : \big( D^* + z \big) \subseteq C^{\beta}_{\mathcal{R}^{-}(z)}(D^* + x) \mbox{ for some } z \in B^{-}_i \mbox{ with } z_d \in 2\hat{k}\mathbb{N} \right\}.
		\end{equation*}
	By the assumptions on $i^*$ and $\epsilon^*$, and the construction of the $\beta$-backbend percolation process, Lemma \ref{lemma iteration case 1} implies that for all $x \in B \big([-l^*,l^*]^{d-2} \times [-2l^*, 2l^*] \times [0, 2s^*] \big)$ with $x_d \in 2\hat{k}\mathbb{Z}_+$,
		\begin{equation}\label{equation final case 1}
			\mathbb{P}_{p-\delta^*} (G_{i^*}^{+}(x)) > 1 - \eta^* \mbox{ and } \mathbb{P}_{p-\delta^*} (G_{i^*}^{-}(x)) > 1 - \eta^*.
		\end{equation}

		\begin{lemma}\label{lemma final case 1}
			$\mathbb{P}_{p - \delta^*} \bigg( \Big| C^{\beta}_{\mathbb{Q}_{2l^*}^2} (D^*) \Big| = \infty \bigg) > 0$.
		\end{lemma}

		Since $\beta$ is $\hat{k}$-cyclic, $s^* > \max \{\beta_0, \ldots, \beta_{\hat{k}-1}\}$, $i^* > 10$, and $\eta^*$ is sufficiently small, the proof of Lemma \ref{lemma final case 1} follows from \eqref{equation final case 1} by using a similar logic as for the proof of Lemma 21 in \citet{bezuidenhout1990critical}.\medskip

		Now, the proof of Proposition \ref{prop percolation at smaller probability} for Case 1 follows from Lemma \ref{lemma final case 1}.\bigskip
		\\
		\noindent\textbf{\textsc{Case} 2}: Suppose $\theta^{\beta}_{\mathbb{S}_t}(p) > 0$ for some $t \in \mathbb{N}$. 
		\\
		Fix $\hat{t} \in \mathbb{N}$ such that $\theta^{\beta}_{\mathbb{S}_{\hat{t}}}(p) > 0$. By a standard argument (see \citet{liggett2012interacting}, Theorem 1.10(d) in Chapter VI for details), this implies that there exists $r^* \in 2\hat{k}\mathbb{N}$ with $2r^* \geq \max \{\hat{t}, \beta_0, \ldots, \beta_{\hat{k}-1}\}$ such that
		\begin{equation}\label{equation from lemma almost 1 prob case 2 partial}
			\mathbb{P}_{p} \Bigg(\Big| C^{\beta}_{\mathbb{S}_{\hat{t}}} \Big( B \big([-r^*, r^*]^{d-1} \times 0 \big) \Big) \Big| = \infty \Bigg) > 1 - \bigg( \frac{\epsilon^*}{2} \bigg)^{(d-1)2^{d-1}}.
		\end{equation}
		For ease of presentation, we denote $B \big([-r^*, r^*]^{d-1} \times 0 \big)$ by $D^*$. Let $t^* := 2r^*$. Since $r^* \in 2\hat{k}\mathbb{N}$ and $2r^* \geq \max \{\hat{t}, \beta_0, \ldots, \beta_{\hat{k}-1}\}$, this implies $t^* \in 2\hat{k}\mathbb{N}$ and $t^* \geq \max \{\hat{t}, \beta_0, \ldots, \beta_{\hat{k}-1}\}$. Furthermore, since $t^* \geq \hat{t}$, $\theta^{\beta}_{\mathbb{S}_{\hat{t}}}(p) > 0$ implies $\theta^{\beta}_{\mathbb{S}_{t^*}}(p) > 0$ and \eqref{equation from lemma almost 1 prob case 2 partial} implies
		\begin{equation}\label{equation from lemma almost 1 prob case 2}
			\mathbb{P}_{p} \Bigg(\Big| C^{\beta}_{\mathbb{S}_{t^*}} ( D^*) \Big| = \infty \Bigg) > 1 - \bigg( \frac{\epsilon^*}{2} \bigg)^{(d-1)2^{d-1}}.
		\end{equation}
		Let us define the following terms involving $\epsilon^*$ and $r^*$ for our next lemma. 
		\begin{enumerate}[$\bullet$]	
			\item Let $\alpha^* := 
			\mathbb{P}_p \Big( B \big([0, 2r^*] \times [-r^*, r^*]^{d-2} \times [2r^* + 2\hat{k}, 2r^* + 4\hat{k}] \big) \subseteq C^{0}_{B \big([0, 2r^*] \times [-r^*, r^*]^{d-2} \times [0, 2r^* + 4\hat{k}] \big)} \Big)$. Clearly, $\alpha^* > 0$.
			
			\item Let $m^* \in \mathbb{N}$ be such that $(1 - \alpha^*)^{m^*} < \frac{\epsilon^*}{2}$. The implication of $m^*$ is that if one performs at least $m^*$ independent trials with probability of success in each trial being $\alpha^*$, the probability that there is at least one success (out of all trials) exceeds $1 - \frac{\epsilon^*}{2}$.				
			
			\item Let $n^* \in \mathbb{N}$ be such that $n^* \geq m^*(8r^* + 16\hat{k} + 1)^d$. The purpose of $n^*$ is to ensure that in any subset of $\mathbb{V}$ having size $n^*$ or larger, there are at least $m^*$ vertices such the $L^{\infty}$-distance between any two of them is at least $(4r^* + 8\hat{k} + 1)$.
		\end{enumerate}

		\begin{lemma}\label{lemma enough vertices case 2}
			There exists $l^* \in 2\hat{k}\mathbb{N}$ with $l^* \geq r^*$ such that for all $v \in \{-1, 1\}^{d-2}$, we have
			\begin{equation*}
				\mathbb{P}_{p} \bigg( \Big| C^{\beta}_{B \big([-l^*, l^*]^{d-1} \times [0, t^*] \big)}(D^*) \cap F^{v}_{(d-1)^{+}}(l^*, t^*) \Big| \geq n^* \bigg) > 1 - \frac{\epsilon^*}{2}.
			\end{equation*}
		\end{lemma}

		The proof of this lemma is relegated to Appendix \ref{appendix Proof of lemma enough vertices case 2}.
		
		Suppose $l^*$ is such that the statement of Lemma \ref{lemma enough vertices case 2} holds. Let $s^* := t^* + 2\hat{k}$. Since $t^* \in 2\hat{k}\mathbb{N}$ and $t^* \geq \max \{\beta_0, \ldots, \beta_{\hat{k}-1}\}$, by the construction of $s^*$, we have $s^* \in 2\hat{k}\mathbb{N}$ and $s^* > \max \{\beta_0, \ldots, \beta_{\hat{k}-1}\}$.

		\begin{lemma}\label{lemma coupling case 2}
			There exists $0 < \delta < p$ such that for all $x \in B \big([-l^*,l^*]^{d-2} \times 0 \times 0 \big)$, we have
			\begin{equation*}
				\mathbb{P}_{p-\delta} 
				\left(
				\begin{aligned}
					& \big( D^* + z \big) \subseteq C^{\beta}_{B \big([-2l^*, 2l^*]^{d-2} \times [-l^*, 3l^*] \times [0, z_d] \big)}(D^* + x)\\
					& \mbox{ for some } z \in B \big([-l^*, l^*]^{d-2} \times [l^*, 2l^*] \times [s^*, 2s^*] \big) \mbox{ with } z_d \in 2\hat{k}\mathbb{N} 
				\end{aligned}
				\right) > 1 - \epsilon^*.
			\end{equation*}
		\end{lemma}

		The proof of this lemma is relegated to Appendix \ref{appendix Proof of lemma coupling case 2}.
		
		Suppose $\delta^*$ is such that the statement of Lemma \ref{lemma coupling case 2} holds. Then, the proof of Proposition \ref{prop percolation at smaller probability} for Case 2 follows from Lemma \ref{lemma coupling case 2}  by using an argument similar to the one which we have used to complete the proof of Proposition \ref{prop percolation at smaller probability} for Case 1 from Lemma \ref{lemma coupling case 1}.

		\subsubsection{Proof of Lemma \ref{lemma enough vertices case 1}}\label{appendix Proof of lemma enough vertices case 1}

		\begin{claim}\label{claim survives prob 0 case 1}
			Let $N \in \mathbb{N}$. Then,
			\begin{equation*}
				\mathbb{P}_{p} \bigg( \big| C^{\beta}_\mathbb{H} (D^*) \big| = \infty \mbox{ and } \Big| \Big( C^{\beta}_{\mathbb{S}_t} (D^*) \cap T(\infty, t) \Big) \Big| < N \mbox{ for infinitely many } t \in \mathbb{N} \bigg) = 0.
			\end{equation*}
		\end{claim}
		
		\begin{claimproof}[\textbf{Proof of Claim \ref{claim survives prob 0 case 1}.}]
			Define the following events:
			\begin{equation*}
				\begin{aligned}
					& B_1 := \bigg\{ \omega \in \Omega : \big| C^{\beta}_\mathbb{H} (D^*) \big| = \infty \mbox{ and } \Big| \Big( C^{\beta}_{\mathbb{S}_t} (D^*) \cap T(\infty, t) \Big) \Big| < N \mbox{ for infinitely many } t \in \mathbb{N} \bigg\},\\
					& B_2 := \left\{ \omega \in \Omega : 
					\begin{aligned}
						& \hspace{1 mm} \Big| \Big(C^{\beta}_{\mathbb{S}_{t'}} (D^*) \cap T(\infty, t') \Big) \Big| > 0 \hspace{1 mm} \forall \hspace{1 mm} t' \in \mathbb{N} \mbox{ and}\\
						& \hspace{1 mm} \Big| \Big( C^{\beta}_{\mathbb{S}_t} (D^*) \cap T(\infty, t) \Big) \Big| < N \mbox{ for infinitely many } t \in \mathbb{N}
					\end{aligned}
					\right\}, \mbox{ and}\\
					& B_3 := \bigg\{ \omega \in \Omega : \big| C^{\beta}_\mathbb{H} (D^*) \big| = \infty \mbox{ and } \Big| \Big(C^{\beta}_{\mathbb{S}_{t'}} (D^*) \cap T(\infty, t') \Big) \Big| = 0 \mbox{ for some } t' \in \mathbb{N} \bigg\}.
				\end{aligned}
			\end{equation*}
			By the constructions of $B_1, B_2$, and $B_3$, along with the construction of the $\beta$-backbend percolation process, we have $B_1 = B_2 \cup B_3$ and $B_2 \cap B_3 = \emptyset$. 
			
			First, we show that $\mathbb{P}_{p} \big( B_3 \big) = 0$. For all $t \in \mathbb{N}$, define 
			\begin{equation*}
				B_3^{t} := \bigg\{ \omega \in \Omega : \big| C^{\beta}_\mathbb{H} (D^*) \big| = \infty \mbox{ and } \Big| \Big(C^{\beta}_{\mathbb{S}_t} (D^*) \cap T(\infty, t) \Big) \Big| = 0 \bigg\}.
			\end{equation*}
			By the assumption for Case 1, we have $\theta^{\beta}_{\mathbb{S}_{t}}(p) = 0$ for all $t \in \mathbb{N}$. Since $D^*$ is a finite set of vertices, by the construction of the $\beta$-backbend percolation process, this implies $\mathbb{P}_{p} \big( B_3^{t} \big) = 0$ for all $t \in \mathbb{N}$. By construction, we have $B_3 = \underset{t \in \mathbb{N}}{\cup} \hspace{1 mm} B_3^{t}$.
			Since $\mathbb{P}_{p} \big( B_3^{t} \big) = 0$ for all $t \in \mathbb{N}$, this implies $\mathbb{P}_{p} \big( B_3 \big) = 0$.
			
			Next, we show that $\mathbb{P}_{p} \big( B_2 \big) = 0$. For all $t \in \mathbb{Z}_+$, let $\mathcal{F}_t$ be the $\sigma$-algebra generated by $\Big\{\big| \big(C^{\beta}_{\mathbb{S}_{t'}} (D^*) \cap T(\infty, t') \big) \big| : 0 \leq t' \leq t \Big\}$.\footnote{Note that $\big| \big( C^{\beta}_{\mathbb{S}_t} (D^*) \cap T(\infty, t) \big) \big|$ is a random variable for all $t \in \mathbb{Z}_+$.} Fix $t \in \mathbb{N}$ and $F \in \mathcal{F}_t$. For a vertex $x \in T(\infty, t)$, define the following event:
			\begin{equation*}
				A_x := \bigg\{ \omega \in \Omega : \langle x, y \rangle \mbox{ is closed whenever } y \in T(\infty, t + 1) \bigg\}.
			\end{equation*}
			Clearly, $\mathbb{P}_{p} (A_x) = (1 - p)^{2^{d-1}}$ for all $x \in T(\infty, t)$. Fix $S \subseteq T(\infty, t)$ such that $|S| < N$. Then,
			\begin{equation*}
				\begin{aligned}
					& \hspace{1 mm} \mathbb{P}_{p} \Bigg( \Big| \Big( C^{\beta}_{\mathbb{S}_{t+1}} (D^*) \cap T(\infty, t + 1) \Big) \Big| = 0 \hspace{2 mm} \biggm| \hspace{2 mm} F \mbox{ and } \Big( C^{\beta}_{\mathbb{S}_t} (D^*) \cap T(\infty, t) \Big) = S \Bigg)
					\\
					= \hspace{1 mm} & \hspace{1 mm} \mathbb{P}_{p} \bigg( \underset{x \in S}{\cap} \hspace{1 mm} A_x \bigg) \hspace{10 mm} \mbox{(by the construction of the $\beta$-backbend percolation process)}\\
					= \hspace{1 mm} & \hspace{1 mm} \underset{x \in S}{\prod} \hspace{1 mm} \mathbb{P}_{p} \big( A_x \big) \hspace{12 mm} \mbox{(since each edge is independent)}\\
					> \hspace{1 mm} & \hspace{1 mm} (1-p)^{2^{d-1}N}.
				\end{aligned}
			\end{equation*}
			This, along with Remark \ref{remark probability inequality}, yields
			\begin{equation}\label{equation atleast one open path}
					\mathbb{P}_{p} \Bigg( \Big| \Big(C^{\beta}_{\mathbb{S}_{t+1}} (D^*) \cap T(\infty, t + 1) \Big) \Big| > 0 \hspace{2 mm} \biggm| \hspace{2 mm} F \mbox{ and } \Big| \Big( C^{\beta}_{\mathbb{S}_t} (D^*) \cap T(\infty, t) \Big) \Big| < N \Bigg) < 1 - (1-p)^{2^{d-1}N}.
			\end{equation}
			Define the sequence of stopping times $(\tau_m)_{m \in \mathbb{N}}$ as follows. For all $\omega \in \Omega$,
			\begin{enumerate}[(i)]
				\item $\tau_1(\omega) := \min \hspace{1 mm} \bigg\{t \in \mathbb{N} : \Big| \Big( C^{\beta}_{\mathbb{S}_t} (D^*) \cap T(\infty, t) \Big) \Big| < N \bigg\}$ and
				
				\item for all $m \in \mathbb{N} \setminus \{1\}$, $\tau_m(\omega) := \min \hspace{1 mm} \bigg\{t > \tau_{m-1}(\omega) : \Big| \Big( C^{\beta}_{\mathbb{S}_t} (D^*) \cap T(\infty, t) \Big) \Big| < N \bigg\}$,
			\end{enumerate}
			where $\min \hspace{1 mm} \emptyset = \infty$. By the definition of $\{\tau_m\}_{m \in \mathbb{N}}$, along with \eqref{equation atleast one open path} and Remark \ref{remark probability inequality}, for all $m \in \mathbb{N}$ and all $F \in \mathcal{F}_{\tau_m}$, we have
			\begin{equation}\label{equation strong markov}
				\mathbb{P}_{p} \bigg( \Big| \Big(C^{\beta}_{\mathbb{S}_{\tau_m + 1}} (D^*) \cap T(\infty, \tau_m + 1) \Big) \Big| > 0 \hspace{2 mm} \Bigm| \hspace{2 mm} F \mbox{ and } \tau_m < \infty \bigg) < 1 - (1-p)^{2^{d-1}N}.
			\end{equation}
			Now,
			\begin{equation}\label{equation prob of infinite t is 0 calculation}
				\begin{aligned}
					\mathbb{P}_{p} (B_2) \hspace{1 mm} = \hspace{1 mm} & \hspace{1 mm} \mathbb{P}_{p} \bigg(\Big| \Big(C^{\beta}_{\mathbb{S}_{t'}} (D^*) \cap T(\infty, t') \Big) \Big| > 0 \hspace{1 mm} \forall \hspace{1 mm} t' \in \mathbb{N} \mbox{ and } \tau_m < \infty \hspace{1 mm} \forall \hspace{1 mm} m \in \mathbb{N} \bigg)
					\\
					= \hspace{1 mm} & \hspace{1 mm} \mathbb{P}_{p} \big(\tau_1 < \infty \big) \times \mathbb{P}_{p} \bigg( \Big| \Big(C^{\beta}_{\mathbb{S}_{\tau_1 + 1}} (D^*) \cap T(\infty, \tau_1 + 1) \Big) \Big| > 0 \hspace{2 mm} \biggm| \hspace{2 mm} \tau_1 < \infty \bigg)
					\\
					& \times \mathbb{P}_{p} \Bigg( \tau_2 < \infty \hspace{2 mm} \biggm| \hspace{2 mm} \tau_1 < \infty, \Big| \Big(C^{\beta}_{\mathbb{S}_{\tau_1 + 1}} (D^*) \cap T(\infty, \tau_1 + 1) \Big) \Big| > 0 \Bigg)
					\\
					& \times \mathbb{P}_{p} 
					\left(
					\Big| \Big(C^{\beta}_{\mathbb{S}_{\tau_2 + 1}} (D^*) \cap T(\infty, \tau_2 + 1) \Big) \Big| > 0 \hspace{3.5 mm} \vrule \hspace{3.5 mm} 
					\begin{aligned}
						& \tau_1 < \infty, \Big| \Big(C^{\beta}_{\mathbb{S}_{\tau_1 + 1}} (D^*) \cap T(\infty, \tau_1 + 1) \Big) \Big| > 0,\\
						& \tau_2 < \infty\\
					\end{aligned}
					\right)
					\\
					& \times \cdots.
				\end{aligned}
			\end{equation}
			Since $p < 1$, we have $1 - (1-p)^{2^{d-1}N} < 1$. This, together with \eqref{equation strong markov} and \eqref{equation prob of infinite t is 0 calculation}, implies $\mathbb{P}_{p} (B_2) = 0$.
						
			Since $B_1 = B_2 \cup B_3$ and $B_2 \cap B_3 = \emptyset$, the facts $\mathbb{P}_{p} (B_2) = 0$ and $\mathbb{P}_{p} (B_3) = 0$ together imply $\mathbb{P}_{p} (B) = 0$. This completes the proof of Claim \ref{claim survives prob 0 case 1}.
		\end{claimproof}

		Let $k^* \in \mathbb{N}$ be such that $\Big(1 - (1 - p^{2\hat{k}})^{k^*} \Big)^{n^*} > 1 - \frac{\epsilon^*}{10}$.\footnote{Since $p$, $n^*$, and $\epsilon^*$ are already fixed, the facts $p > 0$ and $1 - \Big( 1 - \frac{\epsilon^*}{10}\Big)^{\frac{1}{n^*}} > 0$ together imply $k^*$ is well-defined.} Define $A := \Big\{ \omega \in \Omega : \big| C^{\beta}_\mathbb{H} (D^*) \big| = \infty \Big\}$ and $A_k := \bigg\{ \omega \in \Omega : \Big| \Big( C^{\beta}_{\mathbb{S}_t} (D^*) \cap T(\infty, t) \Big) \Big| \geq 2^{d-1}k^*n^* \hspace{1 mm} \forall \hspace{1 mm} t \geq k \bigg\}$ for all $k \in \mathbb{N}$. By construction, we have
		\begin{equation*}
			A_k \underset{k \rightarrow \infty}{\uparrow} \Bigg(A \setminus \bigg\{ \omega : \big| C^{\beta}_\mathbb{H} (D^*) \big| = \infty \mbox{ and } \Big| \Big( C^{\beta}_{\mathbb{S}_t} (D^*) \cap T(\infty, t) \Big) \Big| < 2^{d-1}k^*n^* \mbox{ for infinitely many } t \in \mathbb{N} \bigg\} \Bigg),
		\end{equation*}
		which, together with Claim \ref{claim survives prob 0 case 1}, yields $\mathbb{P}_{p} (A_k) \underset{k \rightarrow \infty}{\uparrow} \mathbb{P}_{p} (A)$. This, along with \eqref{equation from lemma almost 1 prob case 1}, implies that there exists $t_1 \in 2\hat{k}\mathbb{N}$ such that $\mathbb{P}_{p} \big(A_{t_1} \big) > 1 - \frac{1}{2^{2^{d-1}}} \Big( \frac{\epsilon^*}{5} \Big)^{d2^{d-1}}$. This, together with the construction of $A_{t_1}$, yields
		\begin{equation}\label{equation new claim 1}
			\mathbb{P}_{p} \bigg( \Big| \Big( C^{\beta}_{\mathbb{S}_{t_1}} (D^*) \cap T(\infty, t_1) \Big) \Big| \geq 2^{d-1}k^*n^* \bigg) > 1 - \frac{1}{2^{2^{d-1}}} \bigg( \frac{\epsilon^*}{5} \bigg)^{d2^{d-1}}.
		\end{equation}
		Furthermore, by the assumption for Case 1, we have $\theta^{\beta}_{\mathbb{S}_{t_1}}(p) = 0$. Since $D^*$ is a finite set of vertices, this implies 
		\begin{equation*}
			\mathbb{P}_{p} \bigg( \Big| \Big( C^{\beta}_{B \big([-l, l]^{d-1} \times [0, t_1] \big)} (D^*) \cap T(l, t_1) \Big) \Big| \geq 2^{d-1}k^*n^* \bigg) \underset{l \rightarrow \infty}{\uparrow} \mathbb{P}_{p} \bigg( \Big| \Big( C^{\beta}_{\mathbb{S}_{t_1}} (D^*) \cap T(\infty, t_1) \Big) \Big| \geq 2^{d-1}k^*n^* \bigg).
		\end{equation*}
	Since $\Big( \frac{\epsilon^*}{10} \Big)^{2^{d-1}} > \frac{1}{2^{2^{d-1}}} \Big( \frac{\epsilon^*}{5} \Big)^{d2^{d-1}}$, this, along with \eqref{equation new claim 1}, implies that there exists $l(t_1) \in 2\hat{k}\mathbb{N}$ such that
	\begin{equation*}
		\mathbb{P}_{p} \bigg( \Big| \Big( C^{\beta}_{B \big([-l(t_1), l(t_1)]^{d-1} \times [0, t_1] \big)} (D^*) \cap T(l(t_1), t_1) \Big) \Big| \geq 2^{d-1}k^*n^* \bigg) > 1 - \bigg( \frac{\epsilon^*}{10} \bigg)^{2^{d-1}}.
	\end{equation*}
		Define
		\begin{equation*}
			s(l(t_1), t_1) := \underset{s \in 2\hat{k}\mathbb{N}}{\min} \hspace{1 mm} \Bigg\{ s \geq t_1 : \mathbb{P}_{p} \bigg( \Big| \Big( C^{\beta}_{B \big([-l(t_1), l(t_1)]^{d-1} \times [0, s] \big)} (D^*) \cap T(l(t_1), s) \Big) \Big| \geq 2^{d-1}k^*n^* \bigg) \leq 1 - \bigg( \frac{\epsilon^*}{10} \bigg)^{2^{d-1}} \Bigg\},
		\end{equation*}
		where $\min \hspace{1 mm} \emptyset = \infty$.\footnote{Note that by construction, $s(l(t_1), t_1) > t_1$.} Since $p < 1$ and $B(l(t_1), \infty)$ is a one-dimensional cylinder, we have $s(l(t_1), t_1) \in 2\hat{k}\mathbb{N}$.
		
		Construct three sequences of positive integers $(l_k)_{k \in \mathbb{N}}$, $(t_k)_{k \in \mathbb{N}}$, and $(s_k)_{k \in \mathbb{N}}$ as follows. Let $l_1 = \max \big\{ l(t_1), r^* \big\}$ and $s_1 = s(l_1, t_1)$. Suppose $k \geq 1$ and suppose that $l_1, \ldots, l_k$, $t_1, \ldots, t_k$ and $s_1, \ldots, s_k$ are constructed. Choose $t_{k+1} = s_k + 2\hat{k}$, $l_{k+1} = \max \big\{ l(t_{k+1}), (l_k+ 2\hat{k}) \big\}$, and $s_{k+1} = s(l_{k+1},t_{k+1})$. Note that by construction, for all $k \in \mathbb{N}$, we have $l_k, s_k \in 2\hat{k}\mathbb{N}$ with $l_{k+1} \geq l_k +1$ and $s_{k+1} > s_k +1$.\footnote{Since $s_{k+1} > t_{k+1} = s_k + 2\hat{k}$.} Furthermore, note that for all $k \in \mathbb{N}$,
		\begin{subequations}\label{equation part ii 1 final}
			\begin{equation}\label{equation part ii 1 a}
				\mathbb{P}_{p} \bigg( \Big| \Big(C^{\beta}_{B \big([-l_k, l_k]^{d-1} \times [0, t_k] \big)} (D^*) \cap T(l_k,t_k) \Big) \Big| \geq 2^{d-1}k^*n^* \bigg) > 1 - \bigg( \frac{\epsilon^*}{10} \bigg)^{2^{d-1}}, \mbox{ and}
			\end{equation}
			\begin{equation}\label{equation part ii 1}
				\mathbb{P}_{p} \bigg( \Big| \Big(C^{\beta}_{B \big([-l_k, l_k]^{d-1} \times [0, s_k] \big)} (D^*) \cap T(l_k,s_k) \Big) \Big| \geq 2^{d-1}k^*n^* \bigg) \leq 1 - \bigg( \frac{\epsilon^*}{10} \bigg)^{2^{d-1}}.
			\end{equation}
		\end{subequations}
		For all $k \in \mathbb{N}$, define the random variable $N_k : \Omega \longrightarrow \mathbb{N}$ such that for all $\omega \in \Omega$,
		\begin{equation*}
			N_k(\omega) = \Big| \Big( C^{\beta}_{B \big([-l_k, l_k]^{d-1} \times [0, s_k] \big)} (D^*) \cap \big( T(l_k, s_k) \cup F(l_k, s_k) \big) \Big) \Big|.
		\end{equation*}

		\begin{claim}\label{claim for part ii}
			There exists $k_0 \in \mathbb{N}$ such that
			\begin{equation*}
				\mathbb{P}_{p} \Big( N_{k_0} \geq 2^{d-1}k^*n^* + 2(d-1)2^{d-2}n^* \Big) > 1 - \frac{1}{2^{2^{d-1}}} \bigg( \frac{\epsilon^*}{5} \bigg)^{d2^{d-1}}.
			\end{equation*}
		\end{claim}

		\begin{claimproof}[\textbf{Proof of Claim \ref{claim for part ii}.}]
			For all $k \in \mathbb{N}$, let $\mathcal{G}_k$ be the $\sigma$-algebra generated by $\big\{ N_u : 1 \leq u \leq k \big\}$. Fix $k \in \mathbb{N}$ and $G \in \mathcal{G}_{k}$. For a vertex $x \in \big( T(l_k, s_k) \cup F(l_k, s_k) \big)$, define the following event:
			\begin{equation*}
				A_x := \bigg\{ \omega \in \Omega : \langle x, y \rangle \mbox{ is closed whenever } y \in \Big( T \big(l_k + 1, s_k + 1 \big) \cup F \big(l_k + 1, s_k + 1 \big) \Big) \bigg\}.
			\end{equation*}
			Clearly, $\mathbb{P}_{p} (A_x) \geq (1 - p)^{2^d}$ for all $x \in \big( T(l_k, s_k) \cup F(l_k, s_k) \big)$. Fix $S \subseteq \big( T(l_k, s_k) \cup F(l_k, s_k) \big)$ such that $|S| < \big[ 2^{d-1}k^*n^* + 2(d-1)2^{d-2}n^* \big]$. Recall that for all $k \in \mathbb{N}$, we have $l_{k+1} \geq l_k +1$ and $s_{k+1} > s_k +1$. It follows that
			\begin{equation*}
				\begin{aligned}
					& \hspace{1 mm} \mathbb{P}_{p} \Bigg( N_{k+1} = 0 \hspace{2 mm} \biggm| \hspace{2 mm} G \mbox{ and } \Big( C^{\beta}_{B \big([-l_k, l_k]^{d-1} \times [0, s_k] \big)} (D^*) \cap \big( T(l_k, s_k) \cup F(l_k, s_k) \big) \Big) = S \Bigg)\\
					\geq \hspace{1 mm} & \hspace{1 mm} \mathbb{P}_{p} \bigg( \underset{x \in S}{\cap} \hspace{1 mm} A_x \bigg) \hspace{10 mm} \mbox{(by the construction of the $\beta$-backbend percolation process)}\\
					= \hspace{1 mm} & \hspace{1 mm} \underset{x \in S}{\prod} \hspace{1 mm} \mathbb{P}_{p} \big( A_x \big) \hspace{12 mm} \mbox{(since each edge is independent)}\\
					> \hspace{1 mm} & \hspace{1 mm} (1-p)^{2^{2d-1} (k^*n^* + (d-1)n^*)}.
				\end{aligned}
			\end{equation*}
			This, along with Remark \ref{remark probability inequality}, yields
			\begin{equation}\label{equation atleast one open path box}
				\mathbb{P}_{p} \bigg( N_{k+1} > 0 \hspace{2 mm} \Bigm| \hspace{2 mm} G \mbox{ and } N_k < \big[ 2^{d-1}k^*n^* + 2(d-1)2^{d-2}n^* \big] \bigg) < 1 - (1-p)^{2^{2d-1} (k^*n^* + (d-1)n^*)}.
			\end{equation}
		Since $p < 1$, we have $1 - (1-p)^{2^{2d-1} (k^*n^* + (d-1)n^*)} < 1$. This, together with \eqref{equation atleast one open path box} and an argument similar to the one which we use to complete the proof of Claim \ref{claim survives prob 0 case 1} from \eqref{equation atleast one open path}, yields
		\begin{equation}\label{equation survives prob 0 box}
			\mathbb{P}_{p} \Big( \big| C^{\beta}_\mathbb{H} (D^*) \big| = \infty \mbox{ and } N_k < \big[ 2^{d-1}k^*n^* + 2(d-1)2^{d-2}n^* \big] \mbox{ for infinitely many } k \in \mathbb{N} \Big) = 0.
		\end{equation}
		
		The proof of Claim \ref{claim for part ii} follows from \eqref{equation from lemma almost 1 prob case 1} and \eqref{equation survives prob 0 box} by using an argument similar to the one which we use to obtain \eqref{equation new claim 1} from Claim \ref{claim survives prob 0 case 1}.
		\end{claimproof}

		By Claim \ref{claim for part ii}, along with FKG inequality, we have
		\begin{equation*}
			\begin{aligned}
				\frac{1}{2^{2^{d-1}}} \bigg( \frac{\epsilon^*}{5} \bigg)^{d2^{d-1}} \hspace{1 mm} \geq \hspace{1 mm} & \hspace{1 mm} \mathbb{P}_{p} \bigg( \Big| \Big( C^{\beta}_{B \big([-l_{k_0}, l_{k_0}]^{d-1} \times [0, s_{k_0}] \big)} (D^*) \cap T(l_{k_0}, s_{k_0}) \Big) \Big| < 2^{d-1}k^*n^* \bigg)
				\\
				& \times \mathbb{P}_{p} \bigg( \Big| \Big( C^{\beta}_{B \big([-l_{k_0}, l_{k_0}]^{d-1} \times [0, s_{k_0}] \big)} (D^*) \cap F(l_{k_0}, s_{k_0}) \Big) \Big| < 2(d-1)2^{d-2}n^* \bigg),
			\end{aligned}
		\end{equation*}
		which, together with \eqref{equation part ii 1}, yields
		\begin{equation*}
			\bigg( \frac{\epsilon^*}{5} \bigg)^{2(d-1)2^{d-2}} > \mathbb{P}_{p} \bigg( \Big| \Big( C^{\beta}_{B \big([-l_{k_0}, l_{k_0}]^{d-1} \times [0, s_{k_0}] \big)} (D^*) \cap F(l_{k_0}, s_{k_0}) \Big) \Big| < 2(d-1)2^{d-2}n^* \bigg).
		\end{equation*}
		By FKG inequality and the construction of the $\beta$-backbend percolation process, this implies
		\begin{equation}\label{equation part ii 2 final}
			\mathbb{P}_{p} \bigg( \Big| \Big( C^{\beta}_{B \big([-l_{k_0}, l_{k_0}]^{d-1} \times [0, s_{k_0}] \big)} (D^*) \cap F^{v}_{(d-1)^{+}}(l_{k_0}, s_{k_0}) \Big) \Big| < n^* \bigg) < \frac{\epsilon^*}{5} \mbox{ for all } v \in \{-1, 1\}^{d-2}.
		\end{equation}
		Furthermore, by the construction of $l_{k_0}$, $t_{k_0}$, and $s_{k_0}$, we have
		\begin{equation*}
			\bigg( \frac{\epsilon^*}{10} \bigg)^{2^{d-1}} > \mathbb{P}_{p} \bigg( \Big| \Big( C^{\beta}_{B \big([-l_{k_0}, l_{k_0}]^{d-1} \times [0, s_{k_0} - 2\hat{k}] \big)} (D^*) \cap T(l_{k_0}, s_{k_0} - 2\hat{k}) \Big) \Big| < 2^{d-1}k^*n^* \bigg)
		\end{equation*}
		By FKG inequality and the construction of the $\beta$-backbend percolation process, this implies
		\begin{equation}\label{equation part ii 3}
			\mathbb{P}_{p} \bigg( \Big| \Big( C^{\beta}_{B \big([-l_{k_0}, l_{k_0}]^{d-1} \times [0, s_{k_0} - 2\hat{k}] \big)} (D^*) \cap T^u(l_{k_0}, s_{k_0} - 2\hat{k}) \Big) \Big| < k^*n^* \bigg) < \frac{\epsilon^*}{10} \mbox{ for all } u \in \{-1, 1\}^{d-1}.
		\end{equation}

		\begin{claim}\label{claim for part ii 2nd}
			For all $u \in \{-1, 1\}^{d-1}$, we have
			\begin{equation*}
				\mathbb{P}_{p} \bigg( \Big| \Big( C^{\beta}_{B \big([-l_{k_0}, l_{k_0}]^{d-1} \times [0, s_{k_0}] \big)} (D^*) \cap T^u(l_{k_0}, s_{k_0}) \Big) \Big| \geq n^* \bigg) > 1 - \frac{\epsilon^*}{5}.
			\end{equation*}
		\end{claim}
		
		\begin{claimproof}[\textbf{Proof of Claim \ref{claim for part ii 2nd}.}]
			Fix $u \in \{-1, 1\}^{d-1}$. For a vertex $x \in T^u(l_{k_0}, s_{k_0} - 2\hat{k})$, define the following event:
			\begin{equation*}
				A_x := \left\{ \omega \in \Omega : 
				\begin{aligned}
					& \mbox{ there is an open path } ( x^0 = x, \ldots, x^{2\hat{k}} ) \mbox{ in } B \big([-l_{k_0}, l_{k_0}]^{d-1} \times [0, s_{k_0}] \big) \mbox{ from } x,\\
					& \mbox{ where } x^{2i} = x + (0, \ldots, 0, 2i ) \mbox{ for all } i \leq \hat{k} 
				\end{aligned}
				\right\}.	
			\end{equation*}
			By construction, $\{A_x\}_{x \in T^u(l_{k_0}, s_{k_0} - 2\hat{k})}$ are independent events. Furthermore, note that $\mathbb{P}_{p} (A_x) \geq p^{2\hat{k}}$ for all $x \in T^u(l_{k_0}, s_{k_0} - 2\hat{k})$. Fix $S \subseteq T^u(l_{k_0}, s_{k_0} - 2\hat{k})$ such that $|S| \geq k^*n^*$. Make a partition $\big\{ S_1,\ldots, S_{n^*} \big\}$ of $S$ such that $\vert S_i \vert \geq k^*$ for all $i = 1, \ldots, n^*$. For all $i = 1, \ldots, n^*$, define
			\begin{equation*}
				B_i : = \underset{x \in S_i}{\cup} \hspace{1 mm} A_x.
			\end{equation*}
			Clearly, $\mathbb{P}_p(B_i) \geq 1 - (1 - p^{2\hat{k}})^{k^*}$ for all $i = 1, \ldots, n^*$.\footnote{Note that since $\{A_x\}_{x \in T^u(l_{k_0}, s_{k_0} - 2\hat{k})}$ are independent events, $\{A^c_x\}_{x \in T^u(l_{k_0}, s_{k_0} - 2\hat{k})}$ are also independent events.} Now,
			\begin{equation*}
				\begin{aligned}
					& \hspace{1 mm} \mathbb{P}_{p} 
					\left( 
					\begin{aligned}
						\Big| \Big( C^{\beta}_{B \big([-l_{k_0}, l_{k_0}]^{d-1} \times [0, s_{k_0}] \big)} (D^*) \cap & T^u(l_{k_0}, s_{k_0}) \Big) \Big| \\
						& \geq n^*
					\end{aligned} \hspace{3.5 mm} \vrule \hspace{3.5 mm} 
					\begin{aligned}
						\Big( C^{\beta}_{B \big([-l_{k_0}, l_{k_0}]^{d-1} \times [0, s_{k_0} - 2\hat{k}] \big)} (D^*) \cap & T^u(l_{k_0}, s_{k_0} - 2\hat{k}) \Big) \\
						& = S
					\end{aligned}
					\right)
					\\
					\geq \hspace{1 mm} & \hspace{1 mm} \mathbb{P}_{p} \bigg( \underset{i=1}{\overset{n^*}{\cap}} \hspace{1 mm} B_i \bigg) \hspace{10 mm} \mbox{(by the construction of the $\beta$-backbend percolation process)}\\
					\geq \hspace{1 mm} & \hspace{1 mm} \underset{i=1}{\overset{n^*}{\prod}} \hspace{1 mm} \mathbb{P}_{p} ( B_i ) \hspace{12 mm} \mbox{(by FKG inequality)}\\
					\geq \hspace{1 mm} & \hspace{1 mm} \underset{i=1}{\overset{n^*}{\prod}} \hspace{1 mm} \Big( 1 - (1 - p^{2\hat{k}})^{k^*} \Big)\\
					> \hspace{1 mm} & \hspace{1 mm} 1 - \frac{\epsilon^*}{10}. \hspace{12 mm} \mbox{(by the definition of $k^*$)}
				\end{aligned}
			\end{equation*}
			This, together with Remark \ref{remark probability inequality}, implies
			\begin{equation*}
				\begin{aligned}
					& \mathbb{P}_{p} 
					\left( 
					\begin{aligned}
						\Big| \Big( C^{\beta}_{B \big([-l_{k_0}, l_{k_0}]^{d-1} \times [0, s_{k_0}] \big)} (D^*) \cap & T^u(l_{k_0}, s_{k_0}) \Big) \Big| \\
						& \geq n^*
					\end{aligned} \hspace{4 mm} \vrule \hspace{4 mm} 
					\begin{aligned}
						\Big| \Big( C^{\beta}_{B \big([-l_{k_0}, l_{k_0}]^{d-1} \times [0, s_{k_0} - 2\hat{k}] \big)} (D^*) \cap & T^u(l_{k_0}, s_{k_0} - 2\hat{k}) \Big) \Big| \\
						& \geq k^*n^*
					\end{aligned}
					\right)
					\\
					& > \hspace{1 mm} 1 - \frac{\epsilon^*}{10},
				\end{aligned}			 
			\end{equation*}
			which, along with \eqref{equation part ii 3}, yields
			\begin{equation*}
				\mathbb{P}_{p} \bigg( \Big| \Big( C^{\beta}_{B \big([-l_{k_0}, l_{k_0}]^{d-1} \times [0, s_{k_0}] \big)} (D^*) \cap T^u(l_{k_0}, s_{k_0}) \Big) \Big| \geq n^* \bigg) > 1 - \frac{\epsilon^*}{5}.
			\end{equation*}
			This completes the proof of Claim \ref{claim for part ii 2nd}.
		\end{claimproof}

		Let $l^* = l_{k_0}$ and $t^* = s_{k_0}$. By construction, we have $l^*, t^* \in 2\hat{k}\mathbb{N}$ with $l^* \geq r^*$.\footnote{By construction, $l^* = l_{k_0} \geq l_1 \geq r^*$.} This, along with \eqref{equation part ii 2 final} and Claim \ref{claim for part ii 2nd}, completes the proof of Lemma \ref{lemma enough vertices case 1}. \hfill $\qed$

		\subsubsection{Proof of Lemma \ref{lemma coupling case 1}}\label{appendix Proof of lemma coupling case 1}

		Fix an arbitrary $x \in B \big([-l^*,l^*]^{d-2} \times 0 \times 0 \big)$. Define the random variable $\tau : \Omega \longrightarrow 2\hat{k}\mathbb{N} \cup \{\infty\}$ such that for all $\omega \in \Omega$,
		\begin{equation*}
			\tau(\omega) := \min
			\left\{ t \in 2\hat{k}\mathbb{N} : 
			\begin{aligned}
				& \hspace{1 mm} (D^* + z) \subseteq C^{\beta}_{B \big([-2l^*, 2l^*]^{d-2} \times [-l^*, l^* + 2r^*] \times [0, t] \big)}(D^* + x)\\
				& \mbox{ for some } z \in B \big([-l^*, l^*]^{d-2} \times (l^* + r^*) \times t \big)
			\end{aligned}
			\right\},	
		\end{equation*}
		where $\min \hspace{1 mm} \emptyset = \infty$.

		\begin{claim}\label{claim tau is less than S probability}
			$\mathbb{P}_{p} \Big( \tau \in [0,s^*] \Big) > \Big(1 - \frac{\epsilon^*}{5} \Big)^2$.
		\end{claim} 
	
	\begin{claimproof}[\textbf{Proof of Claim \ref{claim tau is less than S probability}.}]
		Choose $v \in \{-1,1\}^{d-2}$ such that $x_iv_i \leq 0$ for all $i = 1, \ldots, d-2$.\footnote{Note that there may be more than one such $v$.} Define the following events:
		\begin{equation*}
			\begin{aligned}
				& A_1 := \left\{ \omega \in \Omega : 
				\begin{aligned}
					& \hspace{1 mm} \exists \hspace{1 mm} S_1 \subseteq \Big( C^{\beta}_{\big(B \big([-l^*, l^*]^{d-1} \times [0, t^*] \big) + x \big)}(D^* + x) \cap \big( F^{v}_{(d-1)^{+}}(l^*,t^*) + x \big) \Big) \mbox{ with } \vert S_1 \vert \geq m^*\\
					& \mbox{ such that for all distinct } y_1, y_2 \in S_1, L^{\infty}(y_1,y_2) \geq 4r^* + 4\hat{k} + 1 
				\end{aligned}
				\right\},\\
				& A_2 := \left\{ \omega \in \Omega : 
				\begin{aligned}
					& \hspace{1 mm} \exists \hspace{1 mm} y \in \Big( C^{\beta}_{\big(B \big([-l^*, l^*]^{d-1} \times [0, t^*] \big) + x \big)}(D^* + x) \cap \big( F^{v}_{(d-1)^{+}}(l^*,t^*) + x \big) \Big) \mbox{ such that}\\
					& \Big(B \big([-r^*, r^*]^{d-2} \times [0, 2r^*] \times [2r^*, 2r^* + 2\hat{k}] \big) + y \Big) \subseteq C^{0}_{\big(B \big([-r^*, r^*]^{d-2} \times [0, 2r^*] \times [0, 2r^* + 2\hat{k}] \big) + y \big)} (\{y\}) 
				\end{aligned}
				\right\}.
			\end{aligned}
		\end{equation*}
		By the construction of $v$, $\big( F^{v}_{(d-1)^{+}}(l^*,t^*) + x \big) \subseteq B \big([-l^*, l^*]^{d-2} \times l^* \times [0, t^*] \big)$. Since $2r^* \geq \max \{\beta_0, \ldots, \beta_{\hat{k}-1}\}$, $s^* = t^* + 2r^* + 2\hat{k}$, and $\big( F^{v}_{(d-1)^{+}}(l^*,t^*) + x \big) \subseteq B \big([-l^*, l^*]^{d-2} \times l^* \times [0, t^*] \big)$, it follows from the construction of $A_2$ that $A_2 \subseteq \big\{ \omega \in \Omega : \tau(\omega) \in [0,s^*] \big\}$. So, we have
		\begin{equation}\label{equation large probability for tau}
			\begin{aligned}
				\mathbb{P}_{p} \Big( \tau \in [0,s^*] \Big) \geq \hspace{1 mm} & \mathbb{P}_{p} \Big( A_2 \Big)\\
				\geq \hspace{1 mm} & \hspace{1 mm} \mathbb{P}_{p} \Big( A_2 \mid A_1 \Big) \times \mathbb{P}_{p} \Big( A_1 \Big)\\
				\geq \hspace{1 mm} & \hspace{1 mm} \mathbb{P}_{p} \Big( A_2 \mid A_1 \Big)\\
				& \times \mathbb{P}_{p} \bigg( A_1 \hspace{2 mm} \biggm| \hspace{2 mm} \Big| \Big( C^{\beta}_{\big(B \big([-l^*, l^*]^{d-1} \times [0, t^*] \big) + x \big)}(D^* + x) \cap \big( F^{v}_{(d-1)^{+}}(l^*,t^*) + x \big) \Big) \Big| \geq n^* \bigg)\\
				& \times \mathbb{P}_{p} \bigg( \Big| \Big( C^{\beta}_{\big(B \big([-l^*, l^*]^{d-1} \times [0, t^*] \big) + x \big)}(D^* + x) \cap \big( F^{v}_{(d-1)^{+}}(l^*,t^*) + x \big) \Big) \Big| \geq n^* \bigg).
			\end{aligned}
		\end{equation}
		Define the mapping $S_1 : A_1 \longrightarrow \mathcal{P}\big( \mathbb{H} \big)$ such that for all $\omega \in A_1$,
		\begin{enumerate}[(i)]
			\item $S_1(\omega) \subseteq \Big( C^{\beta}_{\big(B \big([-l^*, l^*]^{d-1} \times [0, t^*] \big) + x \big)}(D^* + x) \cap \big( F^{v}_{(d-1)^{+}}(l^*,t^*) + x \big) \Big)$, and
			
			\item $\vert S_1(\omega) \vert \geq m^*$, where $L^{\infty}(y_1,y_2) \geq 4r^* + 4\hat{k} + 1$ for all distinct $y_1,y_2 \in S_1(\omega)$.
		\end{enumerate}
		By the definition of $m^*$, $A_1$, $S_1$ and $A_2$, along with Remark \ref{remark probability inequality} and the construction of the $\beta$-backbend percolation process, we have
		\begin{equation}\label{equation large probability for tau cal}
			\begin{aligned}
				& \mathbb{P}_{p} \Big( A_2 \mid A_1 \Big)
				\\
				\geq \hspace{1 mm} & \mathbb{P}_{p} \left( 
				\begin{aligned}
					& \exists \hspace{1 mm} y \in S_1 \mbox{ such that }\\
					& \Big(B \big([-r^*, r^*]^{d-2} \times [0, 2r^*] \times [2r^*, 2r^* + 2\hat{k}] \big) + y \Big) \subseteq C^{0}_{\big(B \big([-r^*, r^*]^{d-2} \times [0, 2r^*] \times [0, 2r^* + 2\hat{k}] \big) + y \big)} (\{y\})
				\end{aligned}
				\hspace{3.5 mm} \vrule \hspace{3.5 mm} A_1 \right)\\
				> \hspace{1 mm} & 1 - \frac{\epsilon^*}{5}.
			\end{aligned}
		\end{equation}
		By the definition of $A_1$ and $n^*$,
		\begin{equation}\label{equation new 1}
			\mathbb{P}_{p} \bigg( A_1 \hspace{2 mm} \biggm| \hspace{2 mm} \Big| \Big( C^{\beta}_{\big(B \big([-l^*, l^*]^{d-1} \times [0, t^*] \big) + x \big)}(D^* + x) \cap \big( F^{v}_{(d-1)^{+}}(l^*,t^*) + x \big) \Big) \Big| \geq n^* \bigg) = 1,
		\end{equation}
	and by Lemma \ref{lemma enough vertices case 1} and the construction of the $\beta$-backbend percolation process,
	\begin{equation*}
		\mathbb{P}_{p} \bigg( \Big| \Big( C^{\beta}_{\big(B \big([-l^*, l^*]^{d-1} \times [0, t^*] \big) + x \big)}(D^* + x) \cap \big( F^{v}_{(d-1)^{+}}(l^*,t^*) + x \big) \Big) \Big| \geq n^* \bigg) > 1 - \frac{\epsilon^*}{5}.
	\end{equation*}
	This, together with \eqref{equation large probability for tau}, \eqref{equation large probability for tau cal}, and \eqref{equation new 1}, completes the proof of Claim \ref{claim tau is less than S probability}.
	\end{claimproof}

		Define the mapping $\zeta : \Omega \longrightarrow B \big([-l^*, l^*]^{d-2} \times (l^* + r^*) \times 2\hat{k}\mathbb{N} \big) \cup \{\infty\}$ such that for all $\omega \in \Omega$,
		\begin{equation*}
			\zeta(\omega) := \begin{cases}	
				z & \mbox{ if } \tau(\omega) < \infty \mbox{ and }\\
				& (D^* + z) \subseteq C^{\beta}_{B \big([-2l^*, 2l^*]^{d-2} \times [-l^*, l^* + 2r^*] \times [0, z_d] \big)}(D^* + x) \mbox{ with } z_d = \tau(\omega);\vspace{9pt}\\
				\infty & \mbox{ if } \tau(\omega) = \infty
			\end{cases}
		\end{equation*}

		\begin{claim}\label{claim translation probability}
			\begin{equation*}
				\mathbb{P}_{p} \left( 
				\begin{aligned}
					& \big( D^* + z \big) \subseteq C^\beta_{\big( B \big([-l^*, l^*]^{d-1} \times [0, s^*] \big) + \zeta \big)} \big( D^* + \zeta \big)\\
					& \mbox{for some } z \in B \big([-l^*, l^*]^{d-2} \times [l^*, 2l^*] \times (\tau + s^*) \big)
				\end{aligned}
				\hspace{3.5 mm} \vrule \hspace{3.5 mm} \tau \in [0, s^*] \right) > \Big( 1 - \frac{\epsilon^*}{5} \Big)^2.
			\end{equation*}
		\end{claim}
	
	\begin{claimproof}[\textbf{Proof of Claim \ref{claim translation probability}.}]
		Define the mapping $u : \Omega \longrightarrow \{-1, 1\}^{d-1} \cup \{\infty\}$ as follows.
		\begin{enumerate}[(i)]
			\item For all $\omega \in \Omega$ with $\zeta(\omega) \in \mathbb{H}$, $u_{d-1}(\omega) = 1$ and $\zeta_i(\omega)u_i(\omega) \leq 0$ for all $i = 1, \ldots, d-2$.
			
			\item For all $\omega \in \Omega$ with $\zeta(\omega) = \infty$, $u(\omega) = \infty$.
		\end{enumerate}
		Since $\beta$ is $\hat{k}$-cyclic, it follows from the definition of $\tau$ and $\zeta$, the construction of the $\beta$-backbend percolation process, Lemma \ref{lemma enough vertices case 1}, and Remark \ref{remark probability inequality} that
		\begin{equation}\label{equation claim 2 equation 1}
			\mathbb{P}_{p} \bigg( \Big| \Big( C^{\beta}_{\big(B \big([-l^*, l^*]^{d-1} \times [0, t^*] \big) + \zeta \big)}(D^* + \zeta) \cap \big( T^{u}(l^*,t^*) + \zeta \big) \Big) \Big| \geq n^* \hspace{2 mm} \biggm| \hspace{2 mm} \tau \in [0, s^*] \bigg) > 1 - \frac{\epsilon^*}{5}.
		\end{equation}
		Define the following event:
		\begin{equation*}
			B_1 := \left\{ \omega \in \Omega : 
			\begin{aligned}
				& \hspace{1 mm} \exists \hspace{1 mm} S_2 \subseteq \Big( C^{\beta}_{\big(B \big([-l^*, l^*]^{d-1} \times [0, t^*] \big) + \zeta \big)}(D^* + \zeta) \cap \big( T^{u}(l^*,t^*) + \zeta \big) \Big) \mbox{ with } \vert S_2 \vert \geq m^*\\
				& \mbox{ such that for all distinct } y_1, y_2 \in S_2, L^{\infty}(y_1,y_2) \geq 4r^* + 4\hat{k} + 1
			\end{aligned}
			\right\}.
		\end{equation*}
		Now, we have
		\begin{equation}\label{equation claim 2 equation 2}
			\begin{aligned}
				& \mathbb{P}_{p} \Big( B_1 \hspace{2 mm} \Bigm| \hspace{2 mm} \tau \in [0, s^*] \Big)\\
				\geq \hspace{1 mm} & \mathbb{P}_{p} \Bigg( B_1 \hspace{2 mm} \biggm| \hspace{2 mm} \tau \in [0, s^*] \mbox{ and } \Big| \Big( C^{\beta}_{\big(B \big([-l^*, l^*]^{d-1} \times [0, t^*] \big) + \zeta \big)}(D^* + \zeta) \cap \big( T^{u}(l^*,t^*) + \zeta \big) \Big) \Big| \geq n^* \Bigg)\\
				& \times \mathbb{P}_{p} \bigg( \Big| \Big( C^{\beta}_{\big(B \big([-l^*, l^*]^{d-1} \times [0, t^*] \big) + \zeta \big)}(D^* + \zeta) \cap \big( T^{u}(l^*,t^*) + \zeta \big) \Big) \Big| \geq n^* \hspace{2 mm} \biggm| \hspace{2 mm} \tau \in [0, s^*] \bigg).
			\end{aligned}
		\end{equation}
		By the definition of $B_1$ and $n^*$, together with Remark \ref{remark probability inequality}, we have 
		\begin{equation*}
			\mathbb{P}_{p} \Bigg( B_1 \hspace{2 mm} \biggm| \hspace{2 mm} \tau \in [0, s^*] \mbox{ and } \Big| \Big( C^{\beta}_{\big(B \big([-l^*, l^*]^{d-1} \times [0, t^*] \big) + \zeta \big)}(D^* + \zeta) \cap \big( T^{u}(l^*,t^*) + \zeta \big) \Big) \Big| \geq n^* \Bigg) = 1.
		\end{equation*}
	This, together with \eqref{equation claim 2 equation 1} and \eqref{equation claim 2 equation 2}, implies 
		\begin{equation}\label{equation claim 2 equation 3}
			\mathbb{P}_{p} \Big( B_1 \hspace{2 mm} \Bigm| \hspace{2 mm} \tau \in [0, s^*] \Big) > 1 - \frac{\epsilon^*}{5}.
		\end{equation}
		
		Define the mapping $S_2 : B_1 \longrightarrow \mathcal{P}\big( \mathbb{H} \big)$ such that for all $\omega \in B_1$,
		\begin{enumerate}[(i)]
			\item $S_2(\omega) \subseteq \Big( C^{\beta}_{\big(B \big([-l^*, l^*]^{d-1} \times [0, t^*] \big) + \zeta \big)}(D^* + \zeta) \cap \big( T^{u}(l^*,t^*) + \zeta \big) \Big)$, and
			
			\item $\vert S_2(\omega) \vert \geq m^*$, where $L^{\infty}(y_1,y_2) \geq 4r^* + 4\hat{k} + 1$ for all distinct $y_1, y_2 \in S_2(\omega)$.
		\end{enumerate}
		By the definition of $\tau$, $\zeta$, $u$, and $B_1$, we have $\big( T^{u}(l^*,t^*) + \zeta \big) \subseteq B \big([-l^*, l^*]^{d-2} \times [l^* + r^*, 2l^* + r^*] \times (\tau + t^*) \big)$ for all $\omega \in B_1$. Define the mappings $\hat{S}_2, \bar{S}_2: B_1 \longrightarrow \mathcal{P}\big( \mathbb{H} \big)$ such that for all $\omega \in B_1$, $\hat{S}_2(\omega) := \{y \in S_2(\omega) : y_{d-1} > 2l^*\}$, and $\bar{S}_2(\omega) := \{y \in S_2(\omega) : y_{d-1} \leq 2l^*\}$. Define the following sub-events of $B_1$:
		\begin{equation*}
			\begin{aligned}
				& \hat{B}_1 := \left\{ \omega \in B_1 : 
				\begin{aligned}
					& \hspace{1 mm} \exists \hspace{1 mm} y \in \hat{S}_2(\omega) \mbox{ such that}\\
					& \hspace{1 mm} \Big(D^* + y + (0, \ldots, 0, -r^* , 2r^* + 2\hat{k}) \Big) \subseteq C^{0}_{\big(B \big([-r^*, r^*]^{d-2} \times [-2r^*, 0] \times [0, 2r^* + 2\hat{k}] \big) + y \big)} (\{y\})
				\end{aligned}
				\right\}, \mbox{and}\\
				& \bar{B}_1 := \left\{ \omega \in B_1 : 
				\begin{aligned}
					& \hspace{1 mm} \exists \hspace{1 mm} y \in \bar{S}_2(\omega) \mbox{ such that}\\
					& \hspace{1 mm} \Big(D^* + y + (0, \ldots, 0, 2r^* + 2\hat{k}) \Big) \subseteq C^{0}_{\big(B \big([-r^*, r^*]^{d-1} \times [0, 2r^* + 2\hat{k}] \big) + y \big)} (\{y\})
				\end{aligned}
				\right\}.
			\end{aligned}
		\end{equation*}
		By the definition of $m^*$, $B_1$, $\hat{B}_1$, and $\bar{B}_1$, along with Remark \ref{remark probability inequality} and the construction of the $\beta$-backbend percolation process, we have
		\begin{equation}\label{equation claim 2 equation 4}
			\mathbb{P}_{p} \Big( \hat{B}_1 \cup \bar{B}_1 \hspace{2 mm} \Bigm| \hspace{2 mm} B_1 \mbox{ and } \tau \in [0, s^*] \Big) > 1 - \frac{\epsilon^*}{5}.
		\end{equation}
		Since $l^* \geq r^*$ (see Lemma \ref{lemma enough vertices case 1} for details) and $s^* = t^* + 2r^* + 2\hat{k}$, we have
		\begin{equation*}
			\begin{aligned}
				& \hspace{1 mm} \mathbb{P}_{p} \left( 
				\begin{aligned}
					& \big( D^* + z \big) \subseteq C^\beta_{\big( B \big([-l^*, l^*]^{d-1} \times [0, s^*] \big) + \zeta \big)} \big( D^* + \zeta \big)\\
					& \mbox{for some } z \in B \big([-l^*, l^*]^{d-2} \times [l^*, 2l^*] \times (\tau + s^*) \big)
				\end{aligned}
				\hspace{3.5 mm} \vrule \hspace{3.5 mm} \tau \in [0, s^*] \right)\\
				\geq \hspace{1 mm} & \hspace{1 mm} \mathbb{P}_{p} \left( 
				\begin{aligned}
					& \big( D^* + z \big) \subseteq C^\beta_{\big( B \big([-l^*, l^*]^{d-1} \times [0, s^*] \big) + \zeta \big)} \big( D^* + \zeta \big)\\
					& \mbox{for some } z \in B \big([-l^*, l^*]^{d-2} \times [l^*, 2l^*] \times (\tau + s^*) \big)
				\end{aligned}
				\hspace{3.5 mm} \vrule \hspace{3.5 mm} B_1 \mbox{ and } \tau \in [0, s^*] \right) \times \mathbb{P}_{p} \Big( B_1 \hspace{2 mm} \bigm| \hspace{2 mm} \tau \in [0, s^*] \Big)\\
				\geq \hspace{1 mm} & \hspace{1 mm} \mathbb{P}_{p} \bigg( \hat{B}_1 \cup \bar{B}_1 \hspace{2 mm} \biggm| \hspace{2 mm} B_1 \mbox{ and } \tau \in [0, s^*] \bigg) \times \mathbb{P}_{p} \Big( B_1 \hspace{2 mm} \bigm| \hspace{2 mm} \tau \in [0, s^*] \Big).
			\end{aligned}
		\end{equation*}
		This, together with \eqref{equation claim 2 equation 3} and \eqref{equation claim 2 equation 4}, completes the proof of Claim \ref{claim translation probability}.
	\end{claimproof}

		By $A_x$, we denote the event $\bigg\{ \omega \in \Omega : \big( D^* + z \big) \subseteq C^{\beta}_{B \big([-2l^*, 2l^*]^{d-2} \times [-l^*, 3l^*] \times [0, z_d] \big)}(D^* + x) \mbox{ for some } z \in B \big([-l^*, l^*]^{d-2} \times [l^*, 2l^*] \times [s^*, 2s^*] \big) \mbox{ with } z_d \in 2\hat{k}\mathbb{N} \bigg\}$. Since $l^* \geq r^*$, by the definition of $\tau$, $\zeta$, and $A_x$, the construction of the lattice structure, and the construction of the $\beta$-backbend percolation process, we have
		\begin{equation}\label{equation lemma 3 final equation}
			\begin{aligned}
				& \hspace{1 mm} \mathbb{P}_{p} \big( A_x \big)\\
				\geq \hspace{1 mm} & \hspace{1 mm} \mathbb{P}_{p} \Big( A_x \hspace{2 mm} \Bigm| \hspace{2 mm} \tau \in [0,s^*] \Big) \times \mathbb{P}_{p} \Big( \tau \in [0,s^*] \Big)\\
				\geq \hspace{1 mm} & \hspace{1 mm} \mathbb{P}_{p} \left( 
				\begin{aligned}
					& \big( D^* + z \big) \subseteq C^\beta_{\big( B \big([-l^*, l^*]^{d-1} \times [0, s^*] \big) + \zeta \big)} \big( D^* + \zeta \big)\\
					& \mbox{for some } z \in B \big([-l^*, l^*]^{d-2} \times [l^*, 2l^*] \times (\tau + s^*) \big)
				\end{aligned}
				\hspace{3.5 mm} \vrule \hspace{3.5 mm} \tau \in [0, s^*] \right) \times \mathbb{P}_{p} \Big( \tau \in [0,s^*] \Big)\\
				> \hspace{1 mm} & \hspace{1 mm} \Big( 1 - \frac{\epsilon^*}{5} \Big)^4 \hspace{10 mm} \mbox{(by Claim \ref{claim tau is less than S probability} and Claim \ref{claim translation probability})}\\
				> \hspace{1 mm} & \hspace{1 mm} 1 - \epsilon^*.
			\end{aligned}	
		\end{equation}
		
		Note that the event $A_x$ depends only on the edges in the set $\Big( B \big([-2l^*, 2l^*]^{d-2} \times [-l^*, 3l^*] \times [0, 2s^*] \big) \Big)^2$. Since this set is finite, it follows that $\mathbb{P}_{p} ( A_x )$ is a continuous function of $p$. Moreover, by construction,  $\mathbb{P}_{p} ( A_x )$ is an increasing function of $p$. Since $\mathbb{P}_{p} ( A_x )$ is an increasing continuous function of $p$, \eqref{equation lemma 3 final equation} implies there exists $\delta_x > 0$ with $\delta_x < p$ such that $\mathbb{P}_{p-\delta_x} ( A_x ) > 1 - \epsilon^*$. 
		
		Let 
		\begin{equation*}
			\delta := \min \hspace{1 mm} \Big\{\delta_x : x \in B \big([-l^*,l^*]^{d-2} \times 0 \times 0 \big) \Big\}.
		\end{equation*}
		Since $B \big([-l^*,l^*]^{d-2} \times 0 \times 0 \big)$ is a finite set of vertices, by the construction of $\delta$, we have $0 < \delta < p$. Furthermore, for all $x \in B \big([-l^*,l^*]^{d-2} \times 0 \times 0 \big)$, since $\mathbb{P}_{p} ( A_x )$ is an increasing continuous function of $p$ and $\mathbb{P}_{p-\delta_x} ( A_x ) > 1 - \epsilon^*$, it follows from the construction of $\delta$ that $\mathbb{P}_{p-\delta} ( A_x ) > 1 - \epsilon^*$ for all $x \in B \big([-l^*,l^*]^{d-2} \times 0 \times 0 \big)$. This completes the proof of Lemma \ref{lemma coupling case 1}.
		\hfill $\qed$

		\subsubsection{Proof of Lemma \ref{lemma transition case 1}}\label{appendix Proof of lemma transition case 1}

		Since $\beta$ is $\hat{k}$-cyclic, Lemma \ref{lemma coupling case 1}, together with the construction of the $\beta$-backbend percolation process, implies that for all $y \in B \big([-l^*,l^*]^{d-2} \times \mathbb{Z} \times \mathbb{Z}_+ \big)$ with $y_d \in 2\hat{k}\mathbb{Z}_+$,
		\begin{subequations}\label{equation transition equation for all}
			\begin{equation}\label{equation transition equation for all 1}
				\mathbb{P}_{p-\delta^*} 
				\left(
				\begin{aligned}
					& \big( D^* + z \big) \subseteq C^{\beta}_{B \big([-2l^*, 2l^*]^{d-2} \times [-l^* + y_{d-1}, 3l^* + y_{d-1}] \times [y_d, z_d] \big)}(D^* + y)\\
					& \mbox{ for some } z \in B \big([-l^*, l^*]^{d-2} \times [l^* + y_{d-1}, 2l^* + y_{d-1}] \times [s^* + y_d, 2s^* + y_d] \big)\\
					& \mbox{ with } z_d \in 2\hat{k}\mathbb{N} 
				\end{aligned}
				\right) > 1 - \epsilon^*,
			\end{equation}
			\begin{equation}\label{equation transition equation for all 2}
				\mathbb{P}_{p-\delta^*} 
				\left(
				\begin{aligned}
					& \big( D^* + z \big) \subseteq C^{\beta}_{B \big([-2l^*, 2l^*]^{d-2} \times [-3l^* + y_{d-1}, l^* + y_{d-1}] \times [y_d, z_d] \big)}(D^* + y)\\
					& \mbox{ for some } z \in B \big([-l^*, l^*]^{d-2} \times [-2l^* + y_{d-1}, -l^* + y_{d-1}] \times [s^* + y_d, 2s^* + y_d] \big)\\
					& \mbox{ with } z_d \in 2\hat{k}\mathbb{N} 
				\end{aligned}
				\right) > 1 - \epsilon^*.
			\end{equation}
		\end{subequations}
		Fix an arbitrary $x \in B \big([-l^*,l^*]^{d-2} \times [-2l^*, 2l^*] \times [0, 2s^*] \big)$ with $x_d \in 2\hat{k}\mathbb{Z}_+$. By $A_x$, we denote the event $\Big\{\omega \in \Omega : \big( D^* + z \big) \subseteq C^{\beta}_{B \big([-2l^*, 2l^*]^{d-2} \times [-3l^*, 4l^*] \times [x_d, z_d] \big)}(D^* + x) \mbox{ for some } z \in B \big([-l^*, l^*]^{d-2} \times [-l^*, 3l^*] \times [2s^*, 4s^*] \big) \mbox{ with } z_d \in 2\hat{k}\mathbb{N} \Big\}$. To prove Lemma \ref{lemma transition case 1}, we distinguish the following five cases.\medskip
		\\
		\noindent\textbf{\textsc{Case} 1}: Suppose $x \in B \big([-l^*,l^*]^{d-2} \times [-2l^*, l^*] \times [s^*, 2s^*] \big)$.
		\\
		It follows from \eqref{equation transition equation for all 1} that $\mathbb{P}_{p-\delta^*} \big( A_x \big) > 1 - \epsilon^*$.\medskip
		\\
		\textbf{\textsc{Case} 2}: Suppose $x \in B \big([-l^*,l^*]^{d-2} \times [l^*, 2l^*] \times [s^*, 2s^*] \big)$.
		\\
		It follows from \eqref{equation transition equation for all 2} that $\mathbb{P}_{p-\delta^*} \big( A_x \big) > 1 - \epsilon^*$.\medskip
	\\
	\textbf{\textsc{Case} 3}: Suppose $x \in B \big([-l^*,l^*]^{d-2} \times [-2l^*, 0] \times [0, s^*] \big)$.
		\\
		It follows from \eqref{equation transition equation for all 1} that
		\begin{equation}\label{equation transition case 3 equation 1}
			\mathbb{P}_{p-\delta^*} 
				\left(
				\begin{aligned}
					& \big( D^* + z \big) \subseteq C^{\beta}_{B \big([-2l^*, 2l^*]^{d-2} \times [-3l^*, 3l^*] \times [x_d, z_d] \big)}(D^* + x)\\
					& \mbox{ for some } z \in B \big([-l^*, l^*]^{d-2} \times [-l^*, 2l^*] \times [s^*, 3s^*] \big) \mbox{ with } z_d \in 2\hat{k}\mathbb{N} 
				\end{aligned}
				\right) > 1 - \epsilon^*.
		\end{equation}
		Define the following events:
		\begin{equation*}
			\begin{aligned}
				& A_1 := \left\{ \omega \in \Omega : 
				\begin{aligned}
					& \hspace{1 mm} \big( D^* + z \big) \subseteq C^{\beta}_{B \big([-2l^*, 2l^*]^{d-2} \times [-3l^*, 3l^*] \times [x_d, z_d] \big)}(D^* + x)\\
					& \mbox{ for some } z \in B \big([-l^*, l^*]^{d-2} \times [-l^*, 2l^*] \times [2s^*, 3s^*] \big) \mbox{ with } z_d \in 2\hat{k}\mathbb{N} 
				\end{aligned}
				\right\},\\
				& A_2 := \left\{ \omega \in \Omega : 
				\begin{aligned}
					& \hspace{1 mm} \big( D^* + z \big) \subseteq C^{\beta}_{B \big([-2l^*, 2l^*]^{d-2} \times [-3l^*, 3l^*] \times [x_d, z_d] \big)}(D^* + x)\\
					& \mbox{ for some } z \in B \big([-l^*, l^*]^{d-2} \times [-l^*, l^*] \times [s^*, 2s^*] \big) \mbox{ with } z_d \in 2\hat{k}\mathbb{N} 
				\end{aligned}
				\right\} \setminus A_1, \mbox{ and}\\
				& A_3 := \left\{ \omega \in \Omega : 
				\begin{aligned}
					& \hspace{1 mm} \big( D^* + z \big) \subseteq C^{\beta}_{B \big([-2l^*, 2l^*]^{d-2} \times [-3l^*, 3l^*] \times [x_d, z_d] \big)}(D^* + x)\\
					& \mbox{ for some } z \in B \big([-l^*, l^*]^{d-2} \times [l^*, 2l^*] \times [s^*, 2s^*] \big) \mbox{ with } z_d \in 2\hat{k}\mathbb{N} 
				\end{aligned}
				\right\} \setminus (A_1 \cup A_2).
			\end{aligned}
		\end{equation*}
		Clearly $A_1, A_2$, and $A_3$ are disjoint. By the construction of $A_1, A_2$, and $A_3$, \eqref{equation transition case 3 equation 1} implies $\mathbb{P}_{p-\delta^*} \Big( A_1 \cup A_2 \cup A_3 \Big) > 1 - \epsilon^*$. Since $A_1, A_2$, and $A_3$ are disjoint, $\mathbb{P}_{p-\delta^*} \Big( A_1 \cup A_2 \cup A_3 \Big) > 1 - \epsilon^*$ implies
		\begin{equation}\label{equation transition case 3 equation 2}
			\mathbb{P}_{p-\delta^*} (A_1) + \mathbb{P}_{p-\delta^*} (A_2) + \mathbb{P}_{p-\delta^*} (A_3) > 1 - \epsilon^*.
		\end{equation}
		Define the mappings $z' : A_2 \longrightarrow B \big([-l^*, l^*]^{d-2} \times [-l^*, l^*] \times [s^*, 2s^*] \big)$ and $z'' : A_3 \longrightarrow B \big([-l^*, l^*]^{d-2} \times [l^*, 2l^*] \times [s^*, 2s^*] \big)$ as follows.
		\begin{enumerate}[(i)]
			\item For all $\omega \in A_2$, $\big( D^* + z'(\omega) \big) \subseteq C^{\beta}_{B \big([-2l^*, 2l^*]^{d-2} \times [-3l^*, 3l^*] \times [x_d, z'_d(\omega)] \big)}(D^* + x)$ with $z'_d(\omega) \in 2\hat{k}\mathbb{N}$.
			
			\item For all $\omega \in A_3$, $\big( D^* + z''(\omega) \big) \subseteq C^{\beta}_{B \big([-2l^*, 2l^*]^{d-2} \times [-3l^*, 3l^*] \times [x_d, z''_d(\omega)] \big)}(D^* + x)$ with $z''_d(\omega) \in 2\hat{k}\mathbb{N}$.
		\end{enumerate}
		By the definition of $z'$ and $z''$, along with Remark \ref{remark probability inequality}, \eqref{equation transition equation for all 1}, and \eqref{equation transition equation for all 2}, we have
		\begin{equation}\label{equation transition equation 2 final}
			\begin{aligned}
				& \mathbb{P}_{p-\delta^*} 
				\left(
				\begin{aligned}
					& \big( D^* + z \big) \subseteq C^{\beta}_{B \big([-2l^*, 2l^*]^{d-2} \times [-3l^*, 4l^*] \times [z'_d, z_d] \big)}(D^* + z')\\
					& \mbox{ for some } z \in B \big([-l^*, l^*]^{d-2} \times [-l^*, 3l^*] \times [2s^*, 4s^*] \big) \mbox{ with } z_d \in 2\hat{k}\mathbb{N} 
				\end{aligned}
				\hspace{3.5 mm} \vrule \hspace{3.5 mm} A_2 \right) > 1 - \epsilon^*, \mbox{ and}\\
				& \mathbb{P}_{p-\delta^*} 
				\left(
				\begin{aligned}
					& \big( D^* + z \big) \subseteq C^{\beta}_{B \big([-2l^*, 2l^*]^{d-2} \times [-3l^*, 4l^*] \times [z''_d, z_d] \big)}(D^* + z'')\\
					& \mbox{ for some } z \in B \big([-l^*, l^*]^{d-2} \times [-l^*, 3l^*] \times [2s^*, 4s^*] \big) \mbox{ with } z_d \in 2\hat{k}\mathbb{N} 
				\end{aligned}
				\hspace{3.5 mm} \vrule \hspace{3.5 mm} A_3 \right) > 1 - \epsilon^*.
			\end{aligned}
		\end{equation}
		
		So, we have
		\begin{equation*}
			\begin{aligned}
				& \hspace{1 mm} \mathbb{P}_{p-\delta^*} \big( A_x \big)\\
				\geq \hspace{1 mm} & \hspace{1 mm} \mathbb{P}_{p-\delta^*} \Big( A_x \cap \big( A_1 \cup A_2 \cup A_3 \big) \Big)\\
				= \hspace{1 mm} & \hspace{1 mm} \mathbb{P}_{p-\delta^*} \big( A_x \cap A_1 \big) + \mathbb{P}_{p-\delta^*} \big( A_x \cap A_2 \big) + \mathbb{P}_{p-\delta^*} \big( A_x \cap A_3 \big) \hspace{10 mm} \mbox{(since $A_1, A_2$, and $A_3$ are disjoint)}\\
				= \hspace{1 mm} & \hspace{1 mm} \mathbb{P}_{p-\delta^*} \big( A_1 \big) + \Big( \mathbb{P}_{p-\delta^*} \big( A_x \mid A_2 \big) \times \mathbb{P}_{p-\delta^*} \big( A_2 \big) \Big) + \Big( \mathbb{P}_{p-\delta^*} \big( A_x \mid A_3 \big) \times \mathbb{P}_{p-\delta^*} \big( A_3 \big) \Big) \hspace{10 mm} \mbox{(since $A_1 \subseteq A_x$)}\\
				\geq \hspace{1 mm} & \hspace{1 mm} \mathbb{P}_{p-\delta^*} \big( A_1 \big) +\\
				& \hspace{1 mm} \left[ \mathbb{P}_{p-\delta^*} 
				\left(
				\begin{aligned}
					& \big( D^* + z \big) \subseteq C^{\beta}_{B \big([-2l^*, 2l^*]^{d-2} \times [-3l^*, 4l^*] \times [z'_d, z_d] \big)}(D^* + z')\\
					& \mbox{ for some } z \in B \big([-l^*, l^*]^{d-2} \times [-l^*, 3l^*] \times [2s^*, 4s^*] \big) \mbox{ with } z_d \in 2\hat{k}\mathbb{N} 
				\end{aligned}
				\hspace{3.5 mm} \vrule \hspace{3.5 mm} A_2 \right) \times \mathbb{P}_{p-\delta^*} \big( A_2 \big) \right] +\\
				& \hspace{1 mm} \left[ \mathbb{P}_{p-\delta^*} 
				\left(
				\begin{aligned}
					& \big( D^* + z \big) \subseteq C^{\beta}_{B \big([-2l^*, 2l^*]^{d-2} \times [-3l^*, 4l^*] \times [z''_d, z_d] \big)}(D^* + z'')\\
					& \mbox{ for some } z \in B \big([-l^*, l^*]^{d-2} \times [-l^*, 3l^*] \times [2s^*, 4s^*] \big) \mbox{ with } z_d \in 2\hat{k}\mathbb{N} 
				\end{aligned}
				\hspace{3.5 mm} \vrule \hspace{3.5 mm} A_3 \right) \times \mathbb{P}_{p-\delta^*} \big( A_3 \big) \right]\\
				> \hspace{1 mm} & \hspace{1 mm} \mathbb{P}_{p-\delta^*} \big( A_1 \big) + \Big( \big( 1 - \epsilon^* \big) \times \mathbb{P}_{p-\delta^*} \big( A_2 \big) \Big) + \Big( \big( 1 - \epsilon^* \big) \times \mathbb{P}_{p-\delta^*} \big( A_3 \big) \Big) \hspace{10 mm} \mbox{(by \eqref{equation transition equation 2 final})}\\
				> \hspace{1 mm} & \hspace{1 mm} \big( 1 - \epsilon^* \big) \times \Big( \mathbb{P}_{p-\delta^*} \big( A_1 \big) + \mathbb{P}_{p-\delta^*} \big( A_2 \big) + \mathbb{P}_{p-\delta^*} \big( A_3 \big) \Big)\\
				> \hspace{1 mm} & \hspace{1 mm} \big( 1 - \epsilon^* \big)^2. \hspace{10 mm} \mbox{(by \eqref{equation transition case 3 equation 2})}
			\end{aligned}
		\end{equation*}
		\noindent\textbf{\textsc{Case} 4}: Suppose $x \in B \big([-l^*,l^*]^{d-2} \times [0, l^*] \times [0, s^*] \big)$.
		\\
		It follows from \eqref{equation transition equation for all 1} that
		\begin{equation}\label{equation transition case 4 equation 1}
			\mathbb{P}_{p-\delta^*} 
				\left(
				\begin{aligned}
					& \big( D^* + z \big) \subseteq C^{\beta}_{B \big([-2l^*, 2l^*]^{d-2} \times [-l^*, 4l^*] \times [x_d, z_d] \big)}(D^* + x)\\
					& \mbox{ for some } z \in B \big([-l^*, l^*]^{d-2} \times [l^*, 3l^*] \times [s^*, 3s^*] \big) \mbox{ with } z_d \in 2\hat{k}\mathbb{N} 
				\end{aligned}
				\right)	> 1 - \epsilon^*.
		\end{equation}
		Define the following events:
		\begin{equation*}
			\begin{aligned}
				& A_1 := \left\{ \omega \in \Omega : 
				\begin{aligned}
					& \hspace{1 mm} \big( D^* + z \big) \subseteq C^{\beta}_{B \big([-2l^*, 2l^*]^{d-2} \times [-l^*, 4l^*] \times [x_d, z_d] \big)}(D^* + x)\\
					& \mbox{ for some } z \in B \big([-l^*, l^*]^{d-2} \times [l^*, 3l^*] \times [2s^*, 3s^*] \big) \mbox{ with } z_d \in 2\hat{k}\mathbb{N} 
				\end{aligned}
				\right\}, \mbox{ and}\\
				& A_2 := \left\{ \omega \in \Omega : 
				\begin{aligned}
					& \hspace{1 mm} \big( D^* + z \big) \subseteq C^{\beta}_{B \big([-2l^*, 2l^*]^{d-2} \times [-l^*, 4l^*] \times [x_d, z_d] \big)}(D^* + x)\\
					& \mbox{ for some } z \in B \big([-l^*, l^*]^{d-2} \times [l^*, 3l^*] \times [s^*, 2s^*] \big) \mbox{ with } z_d \in 2\hat{k}\mathbb{N} 
				\end{aligned}
				\right\} \setminus A_1.
			\end{aligned}
		\end{equation*}
		Clearly $A_1$ and $A_2$ are disjoint. By the construction of $A_1$ and $A_2$, \eqref{equation transition case 4 equation 1} implies $\mathbb{P}_{p-\delta^*} \big( A_1 \cup A_2 \big) > 1 - \epsilon^*$. Since $A_1$ and $A_2$ are disjoint, $\mathbb{P}_{p-\delta^*} \big( A_1 \cup A_2 \big) > 1 - \epsilon^*$ implies
		\begin{equation}\label{equation transition case 4 equation 2}
			\mathbb{P}_{p-\delta^*} (A_1) + \mathbb{P}_{p-\delta^*} (A_2) > 1 - \epsilon^*.
		\end{equation}
		Define the mapping $z' : A_2 \longrightarrow B \big([-l^*, l^*]^{d-2} \times [l^*, 3l^*] \times [s^*, 2s^*] \big)$ such that for all $\omega \in A_2$, we have $\big( D^* + z'(\omega) \big) \subseteq C^{\beta}_{B \big([-2l^*, 2l^*]^{d-2} \times [-l^*, 4l^*] \times [x_d, z'_d(\omega)] \big)}(D^* + x)$ with $z'_d(\omega) \in 2\hat{k}\mathbb{N}$. By the definition of $z'$, together with Remark \ref{remark probability inequality} and \eqref{equation transition equation for all 2}, we have
		\begin{equation}\label{equation transition case 4 equation 4}
			\mathbb{P}_{p-\delta^*} 
				\left(
				\begin{aligned}
					& \big( D^* + z \big) \subseteq C^{\beta}_{B \big([-2l^*, 2l^*]^{d-2} \times [-3l^*, 4l^*] \times [z'_d, z_d] \big)}(D^* + z')\\
					& \mbox{ for some } z \in B \big([-l^*, l^*]^{d-2} \times [-l^*, 3l^*] \times [2s^*, 4s^*] \big) \mbox{ with } z_d \in 2\hat{k}\mathbb{N} 
				\end{aligned}
				\hspace{3.5 mm} \vrule \hspace{3.5 mm} A_2 \right) 
				> 1 - \epsilon^*.
		\end{equation}
	Since $A_1 \subseteq A_x$, and $A_1$ and $A_2$ are disjoint, \eqref{equation transition case 4 equation 2} and \eqref{equation transition case 4 equation 4}, along with an argument similar to the one which we use to complete the proof for Case 3 of this lemma from \eqref{equation transition equation 2 final}, together imply $\mathbb{P}_{p-\delta^*} \big( A_x \big) > \big( 1 - \epsilon^* \big)^2$.\medskip
		\\
		\noindent\textbf{\textsc{Case} 5}: Suppose $x \in B \big([-l^*,l^*]^{d-2} \times [l^*, 2l^*] \times [0, s^*] \big)$.
		\\
		It follows from \eqref{equation transition equation for all 2} that
		\begin{equation}\label{equation transition case 5 equation 1}
			\mathbb{P}_{p-\delta^*} 
				\left(
				\begin{aligned}
					& \big( D^* + z \big) \subseteq C^{\beta}_{B \big([-2l^*, 2l^*]^{d-2} \times [-2l^*, 3l^*] \times [x_d, z_d] \big)}(D^* + x)\\
					& \mbox{ for some } z \in B \big([-l^*, l^*]^{d-2} \times [-l^*, l^*] \times [s^*, 3s^*] \big) \mbox{ with } z_d \in 2\hat{k}\mathbb{N} 
				\end{aligned}
				\right) > 1 - \epsilon^*.
		\end{equation}
		Define the following events:
		\begin{equation*}
			\begin{aligned}
				& A_1 := \left\{ \omega \in \Omega : 
				\begin{aligned}
					& \hspace{1 mm} \big( D^* + z \big) \subseteq C^{\beta}_{B \big([-2l^*, 2l^*]^{d-2} \times [-2l^*, 3l^*] \times [x_d, z_d] \big)}(D^* + x)\\
					& \mbox{ for some } z \in B \big([-l^*, l^*]^{d-2} \times [-l^*, l^*] \times [2s^*, 3s^*] \big) \mbox{ with } z_d \in 2\hat{k}\mathbb{N} 
				\end{aligned}
				\right\}, \mbox{ and}\\
				& A_2 := \left\{ \omega \in \Omega : 
				\begin{aligned}
					& \hspace{1 mm} \big( D^* + z \big) \subseteq C^{\beta}_{B \big([-2l^*, 2l^*]^{d-2} \times [-2l^*, 3l^*] \times [x_d, z_d] \big)}(D^* + x)\\
					& \mbox{ for some } z \in B \big([-l^*, l^*]^{d-2} \times [-l^*, l^*] \times [s^*, 2s^*] \big) \mbox{ with } z_d \in 2\hat{k}\mathbb{N} 
				\end{aligned}
				\right\} \setminus A_1.
			\end{aligned}
		\end{equation*}
		Clearly $A_1$ and $A_2$ are disjoint. By the construction of $A_1$ and $A_2$, \eqref{equation transition case 5 equation 1} implies $\mathbb{P}_{p-\delta^*} \big( A_1 \cup A_2 \big) > 1 - \epsilon^*$. Since $A_1$ and $A_2$ are disjoint, $\mathbb{P}_{p-\delta^*} \big( A_1 \cup A_2 \big) > 1 - \epsilon^*$ implies
		\begin{equation}\label{equation transition case 5 equation 2}
			\mathbb{P}_{p-\delta^*} (A_1) + \mathbb{P}_{p-\delta^*} (A_2) > 1 - \epsilon^*.
		\end{equation}
		Define the mapping $z' : A_2 \longrightarrow B \big([-l^*, l^*]^{d-2} \times [-l^*, l^*] \times [s^*, 2s^*] \big)$ such that for all $\omega \in A_2$, we have $\big( D^* + z'(\omega) \big) \subseteq C^{\beta}_{B \big([-2l^*, 2l^*]^{d-2} \times [-2l^*, 3l^*] \times [x_d, z'_d(\omega)] \big)}(D^* + x)$ with $z'_d(\omega) \in 2\hat{k}\mathbb{N}$. By the definition of $z'$, together with Remark \ref{remark probability inequality} and \eqref{equation transition equation for all 1}, we have
		\begin{equation}\label{equation transition equation 2 final new}
			\begin{aligned}
				\mathbb{P}_{p-\delta^*} 
				\left(
				\begin{aligned}
					& \big( D^* + z \big) \subseteq C^{\beta}_{B \big([-2l^*, 2l^*]^{d-2} \times [-3l^*, 4l^*] \times [z'_d, z_d] \big)}(D^* + z')\\
					& \mbox{ for some } z \in B \big([-l^*, l^*]^{d-2} \times [-l^*, 3l^*] \times [2s^*, 4s^*] \big) \mbox{ with } z_d \in 2\hat{k}\mathbb{N} 
				\end{aligned}
				\hspace{3.5 mm} \vrule \hspace{3.5 mm} A_2 \right) > 1 - \epsilon^*.
			\end{aligned}
		\end{equation}
	Since $A_1 \subseteq A_x$, and $A_1$ and $A_2$ are disjoint, \eqref{equation transition case 5 equation 2} and \eqref{equation transition equation 2 final new}, along with an argument similar to the one which we use to complete the proof for Case 3 of this lemma from \eqref{equation transition equation 2 final}, together imply $\mathbb{P}_{p-\delta^*} \big( A_x \big) > \big( 1 - \epsilon^* \big)^2$.\medskip
	\\
	This completes the proof of Lemma \ref{lemma transition case 1}.
	\hfill $\qed$

		\subsubsection{Proof of Lemma \ref{lemma iteration case 1}}\label{appendix Proof of lemma iteration case 1}

		Fix $i \in \mathbb{N}$ and $x \in B \big([-l^*,l^*]^{d-2} \times [-2l^*, 2l^*] \times [0, 2s^*] \big)$ with $x_d \in 2\hat{k}\mathbb{Z}_+$. For all $z \in B^{+}_i$, define
		\begin{equation*}
			\mathcal{S}^{+}(z) := \bigg( \underset{j=0}{\overset{i-1}{\bigcup}} \hspace{1 mm} B \big([-2l^*, 2l^*]^{d-2} \times [jl^* - 3l^*, jl^* + 4l^*] \times [2js^*, 2js^* + 4s^*] \big) \bigg) \cap \mathcal{R}^{+}(z).
		\end{equation*}
		Since $\beta$ is $\hat{k}$-cyclic, using the result in Lemma \ref{lemma transition case 1} repeatedly, it follows that
		\begin{equation}\label{equation for iterative lemma}
			\mathbb{P}_{p-\delta^*} \bigg( \big( D^* + z \big) \subseteq C^{\beta}_{\mathcal{S}^{+}(z)}(D^* + x) \mbox{ for some } z \in B^{+}_i \mbox{ with } z_d \in 2\hat{k}\mathbb{N} \bigg) > (1 - \epsilon^*)^{2i}.
		\end{equation}
		
		By the definition of $\mathcal{S}^{+}(z)$, we have $\mathcal{S}^{+}(z) \subseteq \mathcal{R}^{+}(z)$ for all $z \in B^{+}_i$. This, together with \eqref{equation for iterative lemma}, implies $\mathbb{P}_{p-\delta^*} \big( G_i^{+}(x) \big) > (1 - \epsilon^*)^{2i}$ which completes the proof of Lemma \ref{lemma iteration case 1}.
		\hfill \qed

		\subsubsection{Proof of Lemma \ref{lemma enough vertices case 2}}\label{appendix Proof of lemma enough vertices case 2}

		Using a similar argument as for Claim \ref{claim for part ii} in Lemma \ref{lemma enough vertices case 1}, it follows from \eqref{equation from lemma almost 1 prob case 2} that there exists $l^* \in 2\hat{k}\mathbb{N}$ with $l^* \geq r^*$ such that
		\begin{equation*}
			\mathbb{P}_{p} \bigg( \Big| \Big( C^{\beta}_{B \big([-l^*, l^*]^{d-1} \times [0, t^*] \big)}(D^*) \cap F(l^*, t^*) \Big) \Big| \geq 2(d-1)2^{d-2}n^* \bigg) > 1 - \bigg( \frac{\epsilon^*}{2} \bigg)^{(d-1)2^{d-1}},
		\end{equation*}
		which, in particular, means
		\begin{equation*}
			\mathbb{P}_{p} \bigg( \Big| \Big( C^{\beta}_{B \big([-l^*, l^*]^{d-1} \times [0, t^*] \big)}(D^*) \cap F(l^*, t^*) \Big) \Big| < 2(d-1)2^{d-2}n^* \bigg) < \bigg( \frac{\epsilon^*}{2} \bigg)^{2(d-1)2^{d-2}}.
		\end{equation*}
		By FKG inequality and the construction of the $\beta$-backbend percolation process, this implies that
		\begin{equation*}
			\mathbb{P}_{p} \bigg( \Big| \Big( C^{\beta}_{B \big([-l^*, l^*]^{d-1} \times [0, t^*] \big)}(D^*) \cap F^{v}_{(d-1)^{+}}(l^*, t^*) \Big) \Big| < n^* \bigg) < \frac{\epsilon^*}{2} \mbox{ for all } v \in \{-1, 1\}^{d-2},
		\end{equation*}
	which completes the proof of Lemma \ref{lemma enough vertices case 2}.
		\hfill $\qed$

		\subsubsection{Proof of Lemma \ref{lemma coupling case 2}}\label{appendix Proof of lemma coupling case 2}

		Fix an arbitrary $x \in B \big([-l^*,l^*]^{d-2} \times 0 \times 0 \big)$. By $A_x$, we denote the event $\bigg\{ \omega \in \Omega : \big( D^* + z \big) \subseteq C^{\beta}_{B \big([-2l^*, 2l^*]^{d-2} \times [-l^*, 3l^*] \times [0, z_d] \big)}(D^* + x) \mbox{ for some } z \in B \big([-l^*, l^*]^{d-2} \times [l^*, 2l^*] \times [s^*, 2s^*] \big) \mbox{ with } z_d \in 2\hat{k}\mathbb{N} \bigg\}$. Choose $v \in \{-1,1\}^{d-2}$ such that $x_iv_i \leq 0$ for all $i = 1, \ldots, d-2$.\footnote{Note that there may be more than one such $v$.} Define the following events:
		\begin{equation*}
			\begin{aligned}
				& A_x^1 := \left\{ \omega \in \Omega : 
				\begin{aligned}
					& \hspace{1 mm} \exists \hspace{1 mm} S_1 \subseteq \Big( C^{\beta}_{\big(B \big([-l^*, l^*]^{d-1} \times [0, t^*] \big) + x \big)}(D^* + x) \cap \big( F^{v}_{(d-1)^{+}}(l^*,t^*) + x \big) \Big) \mbox{ with } \vert S_1 \vert \geq m^*\\
					& \mbox{ such that for all distinct } y_1, y_2 \in S_1, L^{\infty}(y_1,y_2) \geq 4r^* + 8\hat{k} + 1 
				\end{aligned}
				\right\},\\
				& A_x^2 := \left\{ \omega \in \Omega : 
				\begin{aligned}
					& \hspace{1 mm} \exists \hspace{1 mm} y \in \Big( C^{\beta}_{\big(B \big([-l^*, l^*]^{d-1} \times [0, t^*] \big) + x \big)}(D^* + x) \cap \big( F^{v}_{(d-1)^{+}}(l^*,t^*) + x \big) \Big) \mbox{ such that}\\
					& \Big(B \big([-r^*, r^*]^{d-2} \times [0, 2r^*] \times [2r^* + 2\hat{k}, 2r^* + 4\hat{k}] \big) + y \Big) \subseteq C^{0}_{B \big([-r^*, r^*]^{d-2} \times [0, 2r^*] \times [0, 2r^* + 4\hat{k}] \big) + y \big)} (\{y\}) 
				\end{aligned}
				\right\}.
			\end{aligned}
		\end{equation*}
		By the construction of $v$, $\big( F^{v}_{(d-1)^{+}}(l^*,t^*) + x \big) \subseteq B \big([-l^*, l^*]^{d-2} \times l^* \times [0, t^*] \big)$. Since $2r^* \geq \max \{\beta_0, \ldots, \beta_{\hat{k}-1}\}$, $t^* = 2r^*$, $l^* \geq r^*$, $s^* = t^* + 2\hat{k}$, and $\big( F^{v}_{(d-1)^{+}}(l^*,t^*) + x \big) \subseteq B \big([-l^*, l^*]^{d-2} \times l^* \times [0, t^*] \big)$, it follows from the construction of $A_x^2$ that $A_x^2 \subseteq A_x$. So, we have
		\begin{equation}\label{equation lemma coupling case 2 2}
			\begin{aligned}
				\mathbb{P}_{p} \Big( A_x \Big) \geq \hspace{1 mm} & \hspace{1 mm} \mathbb{P}_{p} \Big( A_x^2 \Big)\\
				\geq \hspace{1 mm} & \hspace{1 mm} \mathbb{P}_{p} \Big( A_x^2 \mid A_x^1 \Big) \times \mathbb{P}_{p} \Big( A_x^1 \Big)\\
				\geq \hspace{1 mm} & \hspace{1 mm} \mathbb{P}_{p} \Big( A_x^2 \mid A_x^1 \Big) \times \mathbb{P}_{p} \bigg( A_x^1 \hspace{2 mm} \biggm| \hspace{2 mm} \Big| \Big( C^{\beta}_{\big(B \big([-l^*, l^*]^{d-1} \times [0, t^*] \big) + x \big)}(D^* + x) \cap \big( F^{v}_{(d-1)^{+}}(l^*,t^*) + x \big) \Big) \Big| \geq n^* \bigg)\\
				& \times \mathbb{P}_{p} \bigg( \Big| \Big( C^{\beta}_{\big(B \big([-l^*, l^*]^{d-1} \times [0, t^*] \big) + x \big)}(D^* + x) \cap \big( F^{v}_{(d-1)^{+}}(l^*,t^*) + x \big) \Big) \Big| \geq n^* \bigg).
			\end{aligned}
		\end{equation}
		Define the mapping $S_1 : A_x^1 \longrightarrow \mathcal{P}\big( \mathbb{H} \big)$ such that for all $\omega \in A_x^1$,
		\begin{enumerate}[(i)]
			\item $S_1(\omega) \subseteq \Big( C^{\beta}_{\big(B \big([-l^*, l^*]^{d-1} \times [0, t^*] \big) + x \big)}(D^* + x) \cap \big( F^{v}_{(d-1)^{+}}(l^*,t^*) + x \big) \Big)$, and
			
			\item $\vert S_1(\omega) \vert \geq m^*$, where $L^{\infty}(y_1,y_2) \geq 4r^* + 8\hat{k} + 1$ for all distinct $y_1,y_2 \in S_1(\omega)$.
		\end{enumerate}
		By the definition of $m^*$, $A_x^1$, $S_1$ and $A_x^2$, along with Remark \ref{remark probability inequality} and the construction of the $\beta$-backbend percolation process, we have
		\begin{equation}\label{equation lemma coupling case 2 3}
			\begin{aligned}
				& \mathbb{P}_{p} \Big( A_x^2 \mid A_x^1 \Big)
				\\
				\geq \hspace{1 mm} & \mathbb{P}_{p} \left( 
				\begin{aligned}
					& \exists \hspace{1 mm} y \in S_1 \mbox{ such that }\\
					& \Big(B \big([-r^*, r^*]^{d-2} \times [0, 2r^*] \times [2r^* + 2\hat{k}, 2r^* + 4\hat{k}] \big) + y \Big) \subseteq C^{0}_{\big(B \big([-r^*, r^*]^{d-2} \times [0, 2r^*] \times [0, 2r^* + 4\hat{k}] \big) + y \big)} (\{y\})
				\end{aligned}
				\hspace{3.5 mm} \vrule \hspace{3.5 mm} A_x^1 \right)
				\\
				> \hspace{1 mm} & 1 - \frac{\epsilon^*}{2}.
			\end{aligned}
		\end{equation}
		By the definition of $A_x^1$ and $n^*$,
		\begin{equation}\label{equation new case 2}
			\mathbb{P}_{p} \bigg( A_x^1 \hspace{2 mm} \biggm| \hspace{2 mm} \Big| \Big( C^{\beta}_{\big(B \big([-l^*, l^*]^{d-1} \times [0, t^*] \big) + x \big)}(D^* + x) \cap \big( F^{v}_{(d-1)^{+}}(l^*,t^*) + x \big) \Big) \Big| \geq n^* \bigg) = 1,
		\end{equation}
	and by Lemma \ref{lemma enough vertices case 2} and the construction of the $\beta$-backbend percolation process,
	\begin{equation*}
		\mathbb{P}_{p} \bigg( \Big| \Big( C^{\beta}_{\big(B \big([-l^*, l^*]^{d-1} \times [0, t^*] \big) + x \big)}(D^* + x) \cap \big( F^{v}_{(d-1)^{+}}(l^*,t^*) + x \big) \Big) \Big| \geq n^* \bigg) > 1 - \frac{\epsilon^*}{2}.
	\end{equation*}
This, together with \eqref{equation lemma coupling case 2 2}, \eqref{equation lemma coupling case 2 3}, \eqref{equation new case 2}, yields 
		\begin{equation}\label{equation lemma coupling case 2 4}
			\mathbb{P}_{p} \Big( A_x \Big) > 1 - \epsilon^*.
		\end{equation}

		Note that the event $A_x$ depends only on the edges in the set $\Big( B \big([-2l^*, 2l^*]^{d-2} \times [-l^*, 3l^*] \times [0, 2s^*] \big) \Big)^2$. Since this set is finite, the proof of Lemma \ref{lemma coupling case 2} follows from \eqref{equation lemma coupling case 2 4} by using an argument similar to the one which we use to complete the proof of Lemma \ref{lemma coupling case 1} from \eqref{equation lemma 3 final equation}.
		\hfill $\qed$

		\section{Proof of Proposition \ref{prop bounds two dimensional half slab}}\label{appendix proof of prop bounds two dimensional half slab}
		
		\setcounter{table}{0}
		\setcounter{figure}{0}
		\setcounter{equation}{0}

		Fix an arbitrary $l \in \mathbb{N}$. Using a similar argument as we have suggested for the proof of the ``if'' part of Theorem \ref{theorem iff condition half slabs}, we have $p^{\tilde{\beta}}_c (\mathbb{Q}_l^2) \leq p^{\beta}_c (\mathbb{Q}_l^2)$. We proceed to show that $p^{\beta}_c (\mathbb{Q}_{2l}^2) \leq p^{\tilde{\beta}}_c (\mathbb{Q}_l^2)$. In order to prove it, we show that for all $p \in [0,1]$, $\theta^{\tilde{\beta}}_{\mathbb{Q}_l^2}(p) > 0$ implies $\theta^{\beta}_{\mathbb{Q}_{2l}^2}(p) > 0$. If $p=1$, then there is nothing to prove. Fix an arbitrary $p \in [0,1)$ such that $\theta^{\tilde{\beta}}_{\mathbb{Q}_l^2}(p) > 0$.
		
		Since $\underset{n \rightarrow \infty}{\lim} \hspace{1 mm} \tilde{\beta}_{kn+i} = \beta_i$ for all $i \in \{0, \ldots, k-1\}$, there exists $n^* \in \mathbb{N}$ such that $\tilde{\beta}_{kn+i} = \beta_i$ for all $n \geq n^*$ and all $i \in \{0, \ldots, k-1\}$. Let $t^* \in k\mathbb{N}$ be such that $t^* \geq \max \{\beta_0, \ldots, \beta_{k-1}\}$. Let us define the following events:
		\begin{equation*}
			\begin{aligned}
				& A := \left\{ \omega \in \Omega : 
				\begin{aligned}
					& \hspace{1 mm} \mbox{there is an infinite open $\tilde{\beta}$-backbend path } (x^0, x^1, \ldots) \mbox{ in } \mathbb{Q}_l^2\\
					& \hspace{1 mm} \mbox{from the origin such that } x^j_d = kn^* + t^* \mbox{ for some } j
				\end{aligned}
				\right\}, \mbox{ and}\\
				& B := \left\{ \omega \in \Omega : 
				\begin{aligned}
					& \hspace{1 mm} \mbox{there is an infinite open $\tilde{\beta}$-backbend path } (x^0, x^1, \ldots) \mbox{ in } \mathbb{Q}_l^2\\
					& \hspace{1 mm} \mbox{from the origin such that } x^j_d < kn^* + t^* \mbox{ for all } j
				\end{aligned}
				\right\}.
			\end{aligned}
		\end{equation*}
		By the constructions of $A$ and $B$, we have $\mathbb{P}_p ( A \cup B ) = \theta^{\tilde{\beta}}_{\mathbb{Q}_l^2}(p)$. Furthermore, since $B \big([-l, l]^{d-2} \times \mathbb{Z} \times [0, kn^* + t^*] \big)$ is a one-dimensional cylinder, the fact $p < 1$ implies that $\mathbb{P}_p ( B ) = 0$. Combining the facts that $\theta^{\tilde{\beta}}_{\mathbb{Q}_l^2}(p) > 0$, $\mathbb{P}_p ( A \cup B ) = \theta^{\tilde{\beta}}_{\mathbb{Q}_l^2}(p)$, and $\mathbb{P}_p ( B ) = 0$, we have $\mathbb{P}_p ( A ) > 0$. For all $s \in \mathbb{N}$, let us define the following event:
		\begin{equation*}
			A_s := \left\{ \omega \in \Omega : 
			\begin{aligned}
				& \hspace{1 mm} \mbox{there is an infinite open $\tilde{\beta}$-backbend path } \pi = (x^0, x^1, \ldots) \mbox{ in } \mathbb{Q}_l^2\\
				& \hspace{1 mm} \mbox{from the origin such that } |x^j_{d-1}| < s \mbox{ for all } j \mbox{ with } h^j(\pi) < kn^* + t^*
			\end{aligned}
			\right\}.
		\end{equation*}
		By the constructions of $A$ and $A_s$, it follows that $A_s \underset{s \rightarrow \infty}{\uparrow} A $. Since $\mathbb{P}_p ( A ) > 0$, this implies that there exists $s^* \in \mathbb{N}$ such that 
		\begin{equation}\label{equation positive percolation prob}
			\mathbb{P}_p ( A_{s^*} ) > 0.
		\end{equation}

		\begin{claim}\label{claim subset of beta}
			$A_{s^*} \subseteq \bigg\{ \omega \in \Omega : \Big| C^{\beta}_{B \big([-l, l]^{d-2} \times \mathbb{Z} \times [kn^*, \infty] \big)} \Big( B \big([-l, l]^{d-2} \times [-s^*, s^*] \times kn^* \big) \Big) \Big| = \infty \bigg\}$.
		\end{claim}
		
		\begin{claimproof}[\textbf{Proof of Claim \ref{claim subset of beta}}]
			Fix an arbitrary configuration $\omega^* \in A_{s^*}$. Let $\pi = (x^0, x^1, \ldots)$ be an infinite open $\tilde{\beta}$-backbend path in $\mathbb{Q}_l^2$ from the origin such that $|x^j_{d-1}| < s^*$ for all $j$ with $h^j(\pi) < kn^* + t^*$. Since $B \big([-l, l]^{d-2} \times [-s^*, s^*] \times [0, kn^* + t^*] \big)$ is a finite set of vertices, it follows from the assumptions on $\pi$ that there exists a finite sub-path $\pi_1 = (x^0, \ldots, x^{m^*})$ of $\pi$ from the origin such that $\pi_1$ is in $B \big([-l, l]^{d-2} \times [-s^*, s^*] \times [0, kn^* + t^*] \big)$ and $x^{m^*}_d = kn^* + t^*$. Let $z^*$ be the last vertex of $\pi_1$ such that $z^*_d = kn^*$. Note that by construction, $z^* \in B \big([-l, l]^{d-2} \times [-s^*, s^*] \times kn^* \big)$. Consider the sub-path $\pi^*$ of $\pi$ from $z^*$. Clearly, $\pi^*$ is an infinite open path.
			
			First, we show that $\pi^*$ is in $B \big([-l, l]^{d-2} \times \mathbb{Z} \times [kn^*, \infty] \big)$. We divide $\pi^*$ into two parts. The first part $\pi^*_1$ is the sub-path of $\pi^*$ from $z^*$ to $x^{m^*}$, and the second part $\pi^*_2$ is the sub-path of $\pi^*$ from $x^{m^*}$. 
			By construction, $\pi^*_1$ is in $B \big([-l, l]^{d-2} \times [-s^*, s^*] \times [kn^*, kn^* + t^*] \big)$. 
			Furthermore, because of the facts that $\pi$ is a $\tilde{\beta}$-backbend path in $\mathbb{Q}_l^2$, $x^{m^*}_d = kn^* + t^*$, $\tilde{\beta}_{i} = \beta_i$ for all $i \geq kn^*$, and $t^* \geq \max \{\beta_0, \ldots, \beta_{k-1}\}$, it follows by construction that $\pi^*_2$ is in $B \big([-l, l]^{d-2} \times \mathbb{Z} \times [kn^*, \infty] \big)$. 
			Since $\pi^*_1$ is in $B \big([-l, l]^{d-2} \times [-s^*, s^*] \times [kn^*, kn^* + t^*] \big)$, and $\pi^*_2$ is in $B \big([-l, l]^{d-2} \times \mathbb{Z} \times [kn^*, \infty] \big)$, it follows that $\pi^*$ is in $B \big([-l, l]^{d-2} \times \mathbb{Z} \times [kn^*, \infty] \big)$.
			
			In order to complete the proof of Claim \ref{claim subset of beta}, it remains to show that $\pi^*$ is a $\beta$-backbend path.
			Assume for contradiction that $\pi^*$ is not a $\beta$-backbend path. Since $\pi^*$ is not a $\beta$-backbend path, there exists a vertex $x^*$ in $\pi^*$ such that $x^*_d < h^* - \beta_{h^*}$, where $h^*$ is the record level attained by the path $\pi^*$ till $x^*$. Let $\hat{h}$ be the record level attained by the path $\pi$ till $x^*$. Clearly, $\hat{h} \geq h^* \geq kn^*$. Because $\pi$ is a $\tilde{\beta}$-backbend path, it must be that $x^*_d \geq \hat{h} - \tilde{\beta}_{\hat{h}}$. Since $\hat{h} \geq kn^*$ and $\tilde{\beta}_{i} = \beta_i$ for all $i \geq kn^*$, this implies $x^*_d \geq \hat{h} - \beta_{\hat{h}}$. Combining the facts that $x^*_d < h^* - \beta_{h^*}$ and $x^*_d \geq \hat{h} - \beta_{\hat{h}}$, we have
			\begin{equation}\label{equation beta subpath}
				h^* - \beta_{h^*} > \hat{h} - \beta_{\hat{h}}.
			\end{equation}
			The assumptions on $\beta$ imply that $\beta_l - \beta_m \leq l-m$ for all $l \geq m$. Since $\hat{h} \geq h^*$, this yields $\beta_{\hat{h}} - \beta_{h^*} \leq \hat{h} - h^*$, a contradiction to \eqref{equation beta subpath}. So, it must be that $\pi^*$ is a $\beta$-backbend path. This completes the proof of Claim \ref{claim subset of beta}. 		
		\end{claimproof}

		Now, we complete the proof of Proposition \ref{prop bounds two dimensional half slab}. Because $B \big([-l, l]^{d-2} \times [-s^*, s^*] \times kn^* \big)$ is a finite set, \eqref{equation positive percolation prob} and Claim \ref{claim subset of beta}  together imply that there exists some $x \in B \big([-l, l]^{d-2} \times [-s^*, s^*] \times kn^* \big)$ such that $\mathbb{P}_p \Big( \big| C^{\beta}_{(x + \mathbb{Q}_{2l}^2)} (\{x\}) \big| = \infty \Big) > 0$. Furthermore, since $\beta$ is $k$-cyclic and $x_d = kn^*$, by the construction of the $\beta$-backbend percolation process, we have $\mathbb{P}_p \Big( \big| C^{\beta}_{\mathbb{Q}_{2l}^2} \big| = \infty \Big) = \mathbb{P}_p \Big( \big| C^{\beta}_{(x + \mathbb{Q}_{2l}^2)} (\{x\}) \big| = \infty \Big)$, and hence $\mathbb{P}_p \Big( \big| C^{\beta}_{\mathbb{Q}_{2l}^2} \big| = \infty \Big) > 0$. This completes the proof of Proposition \ref{prop bounds two dimensional half slab}.
		\hfill $\qed$

	\end{appendices}
	
	\bibliographystyle{plainnat}
	\setcitestyle{numbers}
	\bibliography{mybib}

\end{document}